\documentclass{article}
\usepackage[utf8]{inputenc}
\usepackage{tikz, geometry, amsthm, amsmath, graphicx, titlesec, enumitem, amsfonts, amssymb, hyperref, mathtools, authblk, caption, subcaption}

\usetikzlibrary{shapes.misc, positioning}

\newtheorem{theorem}{Theorem}[section]
\newtheorem{lemma}[theorem]{Lemma}
\newtheorem{conjecture}[theorem]{Conjecture}

\theoremstyle{definition}
\newtheorem{example}[theorem]{Example}
\newtheorem{counterexample}[theorem]{Counterexample}
\newtheorem{definition}[theorem]{Definition}

\title{\vspace{-1em}Set-Sequential Labelings of Odd Trees}

\author[1]{Emily N. Eckels\thanks{Corresponding author}}

\author[2]{Ervin Gy\H{o}ri}

\author[1]{Junsheng Liu}

\author[3]{Sohaib Nasir}

\affil[1]{\small{Department of Mathematics, University of Illinois at Urbana-Champaign, Urbana, IL 61801, USA.}

\textit{Email:} \texttt{\href{mailto:eckels2@illinois.edu}{\{eckels2, }\href{mailto:jl49@illinois.edu}{jl49\}}@illinois.edu}}

\affil[2]{\small{Alfr\'{e}d R\'{e}nyi Institute of Mathematics, Budapest, Hungary.

\textit{Email:} \texttt{\href{mailto:gyori.ervin@renyi.hu}{gyori.ervin@renyi.hu}}}}

\affil[3]{Mathematics and Statistics Department, Vassar College, Poughkeepsie, NY 12604, USA. 

\textit{Email:} \texttt{\href{soh.nasir@gmail.com}{soh.nasir@gmail.com}}}

\date{November 6, 2021}

\begin{document}

\maketitle
\vspace{-2.5em}
\begin{abstract}
    A tree \(T\) on \(2^n\) vertices is called set-sequential if the elements in \(V(T)\cup E(T)\) can be labeled with distinct nonzero \((n+1)\)-dimensional \(01\)-vectors such that the vector labeling each edge is the component-wise sum modulo \(2\) of the labels of the endpoints. It has been conjectured that all trees on \(2^n\) vertices with only odd degree are set-sequential (the ``Odd Tree Conjecture''), and in this paper, we present progress toward that conjecture. We show that certain kinds of caterpillars (with restrictions on the degrees of the vertices, but no restrictions on the diameter) are set-sequential. Additionally, we introduce some constructions of new set-sequential graphs from smaller set-sequential bipartite graphs (not necessarily odd trees). We also make a conjecture about pairings of the elements of \(\mathbb{F}_2^n\) in a particular way; in the process, we provide a substantial clarification of a proof of a theorem that partitions \(\mathbb{F}_2^n\) from a 2011 paper \cite{Balister} by Balister et al. Finally, we put forward a result on bipartite graphs that is a modification of a theorem in \cite{Balister}.
    \vspace{1em}
    \vspace{-2em}
\end{abstract}

\section{Introduction}

In 1985, in \cite{Acharya}, Acharya and Hegde introduced the notion of a set-sequential graph, which they defined as a graph \(G\) for which it is possible to assign distinct nonempty subsets of a set \(X\) to the edges and vertices of the graph in such a way that for each \(e\in E(G)\), the label of \(e=uv\) is the symmetric difference of the labels of \(u,v\in V(G)\). Notice that each set \(X\) of size \(n\) has \(2^n\) subsets, and that we can represent these subsets by \(n\)-dimensional \(01\)-vectors: a \(1\) in the \(i\)-th position of the vector indicates membership of the \(i\)-th element of \(X\) in the subset. Under this representation, the symmetric difference becomes addition modulo \(2\), leading to the definition presented in the abstract. The definition of set-sequential in the abstract fits better with our methods of proof throughout this paper, so we have opted to use it throughout.

It is easy to understand what is meant by the term ``set-sequential'' (in other literature, ``strongly set colorable''), but it is difficult to provide an exhaustive list of all graphs or classes of graphs that have this property. In fact, in the more than thirty-five years since Acharya and Hegde's initial introduction of the problem of classifying set-sequential graphs, this problem remains open. Certain classes of graphs are known to be set-sequential, however, including stars on \(2^n\) vertices (it is an easy exercise to check this). Note here that these stars contain only vertices of odd degree. 

Another broad class of set-sequential graphs is the set of paths on \(2^{n}\) vertices for \(n=1\) and \(n\geq 4\). This result was proved initially by Mehta and Vijayakumar in \cite{Mehta}:
\begin{theorem}{\textnormal{(Mehta and Vijayakumar \cite{Mehta})}}
    For any integer \(n\geq 2\), \(\mathbb{G}^n\) is sequentially ternary if and only if \(n\) is neither \(3\) nor \(4\).
\end{theorem}
Here, Mehta and Vijayakumar define \(\mathbb{G}^n\) in the same way as \(\mathbb{F}_2^n\), and their definition of ``sequentially ternary'' is equivalent to our definition of ``set-sequential.'' This result was also proved by Balister et al. in \cite{Balister}, using terminology more similar to what is used in the remainder of this paper:
\begin{theorem}{\textnormal{(Balister et al. \cite{Balister})}}
    The paths \(P_4\) and \(P_8\) are not strongly set colorable while all other paths of the form \(P_{2^{n-1}}\) are strongly set colorable.
\end{theorem}
This raises the following question: Why are the paths \(P_4\) and \(P_8\) not set-sequential? This question is answered in \cite{Balister}, so we do not provide a rigorous explanation, but rather give the following sketch: Consider the path \(P_4\). There are two vertices of even degree here, which means that for two of the vertices, the vectors \(v\) and \(u\) labeling them appear three times in the overall sum, which must be zero. Since we consider addition modulo \(2\), this leads us to the conclusion that \(v=u\), which contradicts our definition of set-sequential. So \(P_4\) is not set-sequential. A similar but more complex argument gives that \(P_8\) is not set-sequential. Not all graphs, then, are set-sequential, and what is more, it seems to be the vertices of even degree that introduce some measure of uncertainty regarding whether a graph is set-sequential or not. In a 2009 paper \cite{Hegde}, Hegde proved that there must indeed be restrictions on vertices of even degree: 
\begin{theorem}\textnormal{(Hegde \cite{Hegde})} 
    If a graph \(G\) (\(p>2\) [here \(p=|V(G)|\)]) has:
    \begin{enumerate}
        \item exactly one or two vertices of even degree \textbf{or}
        
        \item exactly three vertices of even degree, say, \(v_1\), \(v_2\), \(v_3\), and any two of these vertices are adjacent \textbf{or}
        
        \item exactly four vertices of even degree, say, \(v_1\), \(v_2\), \(v_3\), \(v_4\) such that \(v_1v_2\) and \(v_3v_4\) are edges in \(G\), then \(G\) is not strongly-set colorable.
    \end{enumerate}
\end{theorem}

This theorem prompts the following conjecture (which has been put forth before by others, among them Golowich and Kim in \cite{Golowich}, though the precise origin is unknown): 
\begin{conjecture}[Odd Tree Conjecture]
    Any tree on \(2^n\) vertices with only vertices of odd degree is set-sequential.
\end{conjecture}

In particular, the class of caterpillars with vertices of only odd degree has been of interest to some, among them Golowich and Kim \cite{Golowich}, Abhishek \cite{Abhishek} and Agustine (together with Abhishek) \cite{AbhishekII}. A 2012 paper by Abhishek and Agustine \cite{AbhishekII} presented results for some classes of graphs of diameter \(4\) - in particular, caterpillars with vertices of certain odd degrees. Following that, in 2013, Abhishek \cite{Abhishek} extended those results to certain caterpillars of diameter \(5\). Most recently, in a 2020 paper \cite{Golowich}, Golowich and Kim set forth results that show that several classes of graphs are set-sequential, including odd caterpillars of diameter at most \(18\). We were able to show that another larger class of caterpillars is set-sequential, without making the assumption that the caterpillars are ``small enough.'' 

 Further results on set-sequential trees can be found in a 2011 paper by Balister et al. \cite{Balister}. In addition to the result that all paths except \(P_4\) and \(P_8\) are set-sequential, the authors proved that bipartite graphs can be connected in such a way that produces larger set-sequential graphs. Balister et al. also put forth a conjecture that partitions \(\mathbb{F}_2^n\) (the field of \(n\)-dimensional \(01\)-vectors under addition modulo \(2\)) in a specific way and provided a proof of one case of that conjecture. We as well prove some results for bipartite graphs and address another case of their partitioning conjecture, noting that this conjecture can be used to aid us in our goal of proving the Odd Tree Conjecture: The vertex and edge labels in a set-sequential labeling are in fact elements of \(\mathbb{F}_2^n\).

In addition to these results, we note that the proof of Theorem 4 given in \cite{Balister} treats only one of several cases of that theorem and so is incomplete. In this paper, we give a more explicit generalization of the technique of that proof and so provide a rigorous proof of the theorem.

\section{Odd Tree Conjecture}
One approach to attempting to prove the Odd Tree Conjecture is to try to find a labeling for each odd tree on \(2^n\) vertices for all \(n\), perhaps with the aid of a computer program. This idea does have merit, and we used it, along with the fact stated in the introduction that stars are set-sequential, to show that all odd trees on \(8\) vertices are set-sequential. 

    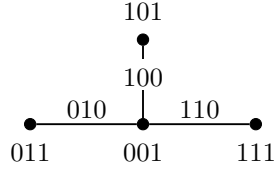
\begin{figure}[h]
    \begin{center}
        \begin{tikzpicture}[scale=0.75]
        \begin{scope}[every node/.style={circle,fill=black,inner sep=0pt, minimum size = 1.5mm,draw}]
            \node (A) [label=below:{001}] at (0,0){};
            \node (B) [label=below:{011}] at (-2,0){};
            \node (C) [label=above:{101}] at (0,1.5){};
            \node (D) [label=below:{111}] at (2,0){};
        \end{scope}
        \begin{scope}[line width = 0.25mm]
            \path (A) edge node [label=above:{010},yshift=-0.5em] {} (B);
            \path (A) edge node [minimum size=4mm,fill=white,anchor=center, pos=0.6,label={[yshift=-1.4em]:100}] {} (C);
            \path (A) edge node [label=above:{110},yshift=-0.5em] {} (D);
        \end{scope}
        \end{tikzpicture}\\
    \end{center}
    \caption{Set-sequential labeling of the only odd tree on 4 vertices} \label{fig:odd 4}
    \end{figure}
    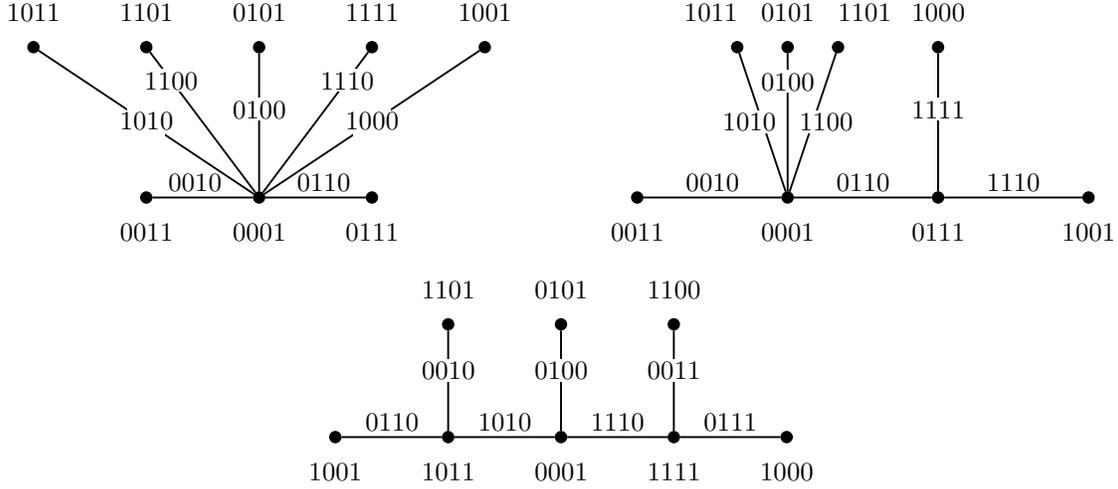
\begin{figure}[h]
    \begin{center}
        \begin{tikzpicture}
        \begin{scope}[every node/.style={circle,fill=black,inner sep=0pt, minimum size = 1.5mm,draw}]
            \node (A) [style={fill=black},label=below:{0011}] at (-1.5,0) {};
            \node (B) [style={fill=black},label=below:{0001}] at (0,0) {};
            \node (C) [style={fill=black},label=below:{0111}] at (1.5,0) {};
            \node (D) [label={above:1011}] at (-3,2) {};
            \node (E) [label={above:1101}] at (-1.5,2) {};
            \node (F) [style={fill=black},label=above:{0101}] at (0,2) {};
            \node (G) [label={above:1111}] at (1.5,2) {};
            \node (H) [label={above:1001}] at (3,2) {};
        \end{scope}
        \begin{scope}[line width = 0.25mm]
            \path (A) edge node  [label=above:{0010},yshift=-0.5em,xshift=-0.3em] {} (B);
            \path (B) edge node [label={[above,yshift=-0.5em,xshift=0.3em]:0110}] {} (C);
            \path (B) edge node [minimum size=5mm,fill=white,anchor=center, pos=0.5,label={[yshift=-1.4em]:1010}]  {} (D);
            \path (B) edge node [minimum size=3mm,fill=white,anchor=center, pos=0.8,label={[yshift=-1.2em]:1100}] {} (E);
            \path (B) edge node [minimum size=3mm,fill=white,anchor=center, pos=0.6,label={[yshift=-1.2em]:0100}]{} (F);
            \path (B) edge node [minimum size=3mm,fill=white,anchor=center, pos=0.8,label={[yshift=-1.2em]:1110}] {} (G);
            \path (B) edge node [minimum size=5mm,fill=white,anchor=center, pos=0.5,label={[yshift=-1.4em]:1000}] {} (H);
        \end{scope}
        \end{tikzpicture}
        \hspace{3em}
        \begin{tikzpicture}
        \begin{scope}[every node/.style={circle,fill=black,inner sep=0pt, minimum size = 1.5mm,draw}]
            \node (A) [style={fill=black},label=below:{0011}] at (-3,0) {};
            \node (B) [style={fill=black},label=below:{0001}] at (-1,0) {};
            \node (C) [style={fill=black},label=below:{0111}] at (1,0) {};
            \node (D) [style={fill=black},label=below:{1001}] at (3,0) {};
            \node (E) [label = {[above]:1000}] at (1,2) {};
            \node (F) [label = {[above,xshift=-1em]:1011}] at (-1.67,2) {};
            \node (G) [style={fill=black},label=above:{0101}] at (-1,2) {};
            \node (H) [label = {[above,xshift=1em]:1101}] at (-0.33,2) {};
        \end{scope}
        \begin{scope}[line width = 0.25mm]
            \path (A) edge node [label=above:{0010},yshift=-0.5em] {} (B);
            \path (B) edge node [label={[above]:0110},yshift=-0.5em] {} (C);
            \path (C) edge node [label = {[above]:1110},yshift=-0.5em] {} (D);
            \path (B) edge node [minimum size=3mm,fill=white,anchor=center, pos=0.5,label={[yshift=-1.15em,xshift=-0.5em]:1010}] {} (F);
            \path (B) edge node [minimum size=3mm,fill=white,anchor=center, pos=0.8,label={[yshift=-1.2em]:0100}] {} (G);
            \path (B) edge node [minimum size=3mm,fill=white,anchor=center, pos=0.5,label={[yshift=-1.15em,xshift=0.5em]:1100}] {} (H);
            \path (C) edge node [minimum size=3mm,fill=white,anchor=center, pos=0.6,label={[yshift=-1.2em]:1111}] {} (E);
        \end{scope}
        \end{tikzpicture}
        \begin{tikzpicture}
        \begin{scope}[every node/.style={circle,fill=black,inner sep=0pt, minimum size = 1.5mm,draw}]
            \node (A) [label=below:{1001}] at (-3,0) {};
            \node (B) [style={fill=black},label=below:{1011}] at (-1.5,0) {};
            \node (C) [style={fill=black},label=below:{0001}] at (0,0) {};
            \node (D) [style={fill=black},label=below:{1111}] at (1.5,0) {};
            \node (E) [label=below:{1000}] at (3,0) {};
            \node (F) [label=above:{1101}] at (-1.5,1.5) {};
            \node [style={fill=black},label=above:{0101}] (G) at (0,1.5) {};
            \node (H) [label=above:{1100}] at (1.5,1.5) {};
        \end{scope}
        \begin{scope}[line width = 0.25mm]
            \path (A) edge node [label=above:{0110},yshift=-0.5em] {} (B);
            \path (B) edge node [label={[above]:1010},yshift=-0.5em] {} (C);
            \path (C) edge node [label={[above]:1110},yshift=-0.5em] {} (D);
            \path (D) edge node [label=above:{0111},yshift=-0.5em] {} (E);
            \path (B) edge node [minimum size=3mm,fill=white,anchor=center, pos=0.6,label={[yshift=-1.15em]:0010}] {} (F);
            \path (C) edge node [minimum size=3mm,fill=white,anchor=center, pos=0.6,label={[yshift=-1.15em]:0100}] {} (G);
            \path (D) edge node [minimum size=3mm,fill=white,anchor=center, pos=0.6,label={[yshift=-1.15em]:0011}] {} (H);
        \end{scope}
        \end{tikzpicture}
    \end{center}
    \caption{Set-sequential labelings of the three odd trees on 8 vertices} \label{fig:odd 8}
    \end{figure}

\begin{lemma}
\label{odd tree}
    The Odd Tree Conjecture is true for \(n\leq 3\). 
\end{lemma}
\begin{proof}
    This is trivial for \(n=1\). For \(n=2\) and \(n=3\), it is sufficient to exhibit a set-sequential labeling of the edges and vertices for each odd tree on \(2\) and \(3\) vertices. For \(n=2\), there is only one odd tree on \(4\) vertices, and we may label it as in Figure \ref{fig:odd 4}. For \(n=3\), there are \(3\) odd trees on \(8\) vertices, and we may label those as in Figure \ref{fig:odd 8}.

\end{proof}
As \(n\) grows large (greater than \(3\) in fact), however, both the number of odd trees on \(2^n\) vertices and the number of possible labelings become too great to check exhaustively due to limits in computational speed. 
We do know that at least some odd trees on \(16\) vertices are set-sequential, but this is not due to a computer program but rather to methods set forth in the remainder of this section. We have studied in earnest three main operations which we present here: constructing caterpillars, splicing smaller graphs together, and using a method introduced to us in \cite{Balister}.

\subsection{Constructing Caterpillars}

As mentioned in the introduction, a 2020 paper by Golowich and Kim \cite{Golowich} presents some results regarding the set-sequentialness of caterpillars: Let \(C\) be an odd caterpillar of diameter \(k\). One result in \cite{Golowich} is that \(C\) is set-sequential if \(k\) is at most \(18\), and another is that \(C\) is set-sequential if \(2^{k-1}\leq |V(C)|\). We add to these results that an even larger class of odd caterpillars is set-sequential. Our result places some restrictions on the possible degrees of vertices along the path that serves as the ``bone'' of the caterpillar but has no limitation on the diameter.

Recall from the introduction the result by Balister et al. in \cite{Balister} that all paths of the form \(P_{2^{n}}\) except for \(P_4\) and \(P_8\) are set-sequential. Start with a path \(P_{2^k}\), with \(k>4\). From this we will construct a larger caterpillar with vertices of certain degrees that is also set-sequential. We allow the path \(P_{2^k}\) to become the ``bone'' of the caterpillar and add pendent edges in such a manner that each new vertex and the edge attaching it to the path correspond to the preceding or subsequent vertex and edge in the path.
\begin{lemma}
\label{small caterpillar}
    Let \(C_{k,3}\) denote the caterpillar on \(k\) vertices with only vertices of degrees 1 and 3. For \(n=2,3\) and \(n>4\), \(C_{2^n,3}\) is set-sequential.
\end{lemma}
\begin{proof}
    Lemma \ref{odd tree} gives that \(C_{2^2,3}\) and \(C_{2^3,3}\) are set-sequential. To show the cases where \(n>4\), we utilize Theorem 1 from \cite{Balister}: Paths of the form \(P_{2^{m}}\) for \(m\geq 4\) have a labeling.\\
    
    Observe that we may write \(C_{2^n,3}\) as the path \(P_{2^{n-1}}\) with one edge connecting each interior vertex to a single additional vertex and two edges connecting two vertices to the last vertex in the path. Suppose \(n> 4\). Take some labeling \(w_1,f_1,w_2,f_2,\ldots,f_{2^{n-1}-1},w_{2^{n-1}}\) with \(n\)-dimensional vectors for the path \(P_{2^{n-1}}\) (denoted in blue). Append a \(0\) to each vector to form vectors of dimension \((n+1)\), noting that this does not change the validity of the labeling. Label the pendent edges and outer vertices with \((n+1)\)-dimensional vectors 
    \(v_1,e_1,\ldots,v_{2^{n-1}-1},e_{2^{n-1}-1},v_{2^{n-1}},e_{2^{n-1}}\), as shown in Figure \ref{fig:C_{k,3}}. 
    \begin{figure}
    \begin{center}
        \begin{tikzpicture}
        \begin{scope}[every node/.style={circle,fill=blue,inner sep=0pt, minimum size = 1.5mm,draw}]
            \node (A) [label=below:{\textcolor{blue}{\(w_1\)}0}] at (0,0) {};
            \node (B) [label=below:{\textcolor{blue}{\(w_2\)}0}] at (1.5,0) {};
            \node (C) [label=below:{\textcolor{blue}{\(w_3\)}0}] at (3,0) {};
            \node (D) [label=below:{\textcolor{blue}{\(w_4\)}0}] at (4.5,0) {};
            \node (E) [label=below:{\textcolor{blue}{\(w_5\)}0}] at (6,0) {};
            \node (F) [label=below:{\textcolor{blue}{\(w_6\)}0}] at (7.5,0) {};
            \node (G) [label=below:{\textcolor{blue}{\(w_7\)}0}] at (9,0) {};
            \node (H) [label=below:{\textcolor{blue}{\(w_8\)}0}] at (10.5,0) {};
            \node (I) [label={[below]:\textcolor{blue}{\(w_{2^{n-1}}\)}0}] at (14.5,0) {};
            \node(inv1) at (12,0) {};
            \node(inv2) at (13,0) {};
        \end{scope}
        \begin{scope}[every node/.style={circle,fill=black,inner sep=0pt, minimum size = 1.5mm,draw}]
            \node (J) [label=below:{\(v_{2^{n-1}}\)}]at (16,0) {};
            \node (K) [label=above:{\(v_{1}\)}] at (1.5,2) {};
            \node (L) [label=above:{\(v_{2}\)}] at (3,2) {};
            \node (M) [label=above:{\(v_{3}\)}] at (4.5,2) {};
            \node (N) [label=above:{\(v_{4}\)}] at (6,2) {};
            \node (O) [label=above:{\(v_{5}\)}] at (7.5,2) {}; 
            \node (P) [label=above:{\(v_{6}\)}] at (9,2) {};
            \node (Q) [label=above:{\(v_{7}\)}] at (10.5,2) {};
            \node (R) [label={[above,yshift=-1em]:\(v_{2^{n-1}-1}\)}] at (14.5,2) {};
        \end{scope}
        \begin{scope}[line width = 0.25mm]
            \path (A) edge node [label=above:{\textcolor{blue}{\(f_1\)}0}]{} (B);
            \path (B) edge node [label=above:{\textcolor{blue}{\(f_2\)}0}]{} (C);
            \path (C) edge node [label=above:{\textcolor{blue}{\(f_3\)}0}]{} (D);
            \path (D) edge node [label=above:{\textcolor{blue}{\(f_4\)}0}]{} (E);
            \path (E) edge node [label=above:{\textcolor{blue}{\(f_5\)}0}]{} (F);
            \path (F) edge node [label=above:{\textcolor{blue}{\(f_6\)}0}]{} (G);
            \path (G) edge node [label=above:{\textcolor{blue}{\(f_7\)}0}]{} (H);
            \path (H) edge node [label=above:{\textcolor{blue}{\(f_8\)}0}] {} (inv1);
            \path [dashed] (inv1) edge node {} (inv2);
            \path (inv2) edge node [label=above:{\textcolor{blue}{\(f_{2^{n-1}-1}\)}0}]{} (I);
            \path (I) edge node [label=above:{\(e_{2^{n-1}}\)}]{} (J);
            \path (K) edge node [label=right:{\(e_{1}\)},yshift=1em]{} (B);
            \path (L) edge node [label=right:{\(e_{2}\)},yshift=1em]{} (C);
            \path (M) edge node [label=right:{\(e_{3}\)},yshift=1em]{} (D);
            \path (N) edge node [label=right:{\(e_{4}\)},yshift=1em]{} (E);
            \path (O) edge node [label=right:{\(e_{5}\)},yshift=1em]{} (F);
            \path (P) edge node [label=right:{\(e_{6}\)},yshift=1em]{} (G);
            \path (Q) edge node [label=right:{\(e_{7}\)},yshift=1em]{} (H);
            \path (R) edge node [label=right:{\(e_{2^{n-1}-1}\)},yshift=1em]{} (I);
        \end{scope}
        \end{tikzpicture}
    \end{center}
    \caption{Labeling of the pendent edges and vertices} \label{fig:C_{k,3}}
    \end{figure}
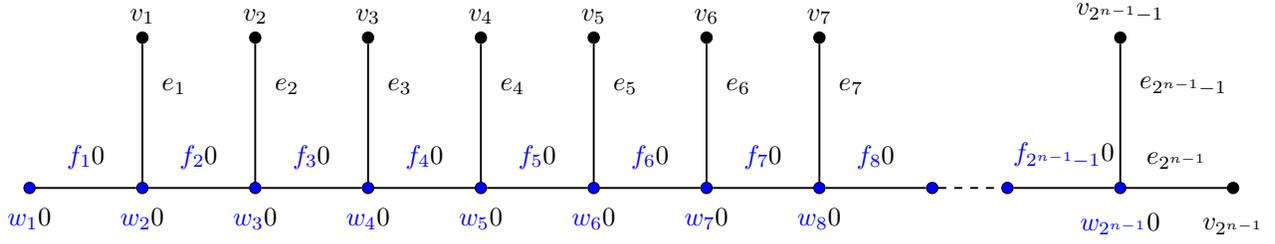
    Since \(1+1=0\) under addition modulo \(2\), \(w_i1+f_i1=w_i0+f_i0=w_{i+1}0\). So let \(v_i=w_i1\) and \(e_{i}=f_i1\). This is a good labeling for \(v_i\) and \(e_{i}\) for \(i=1,\ldots,2^{n-1}-1\). What remains is to label \(v_{2^{n-1}}\) and \(e_{2^{n-1}}\). Observe that we did not use the vector \(w_{2^{n-1}}1\) or the \((n+1)\)-dimensional vector \(0\cdots01\) for any labelings so far. Since \(w_{2^{n-1}}0+w_{2^{n-1}}1=0\cdots01\), let \(v_{2^{n-1}}=w_{2^{n-1}}1\) and \(e_{2^{n-1}}=0\cdots01\).
\end{proof}
Another more general result may be obtained in the same manner:
\begin{theorem}
    Take \(n,k>4\) with \(n\geq k\), and consider the path \(P_{2^{k-1}}\). Join \(2^{n-k+1}-1\) pendent edges and vertices to each of the \(2^{k-1}-2\) interior vertices of the path, and join an additional \(2\cdot(2^{n-k+1}-1)\) pendent edges and vertices to any one of the vertices in \(P_{2^{k-1}}\). The caterpillar on \(2^n\) vertices constructed in this way is set-sequential.
\end{theorem}
Note that choosing \(k=n\) and choosing to join the pendent edges and vertices to the last vertex in the path gives Lemma \ref{small caterpillar}. 
\begin{proof}
    As in the proof of Lemma \ref{small caterpillar}, we use that paths of the form \(P_{2^{m}}\) have a labeling for \(m\geq 4\). Note here that in the construction we propose, we have
    \[2^{k-1} + (2^{k-1}-2)\cdot(2^{n-k+1}-1) + 2\cdot(2^{n-k+1}-1)=2^{k-1}+2^{n}-2^{k-1}-2^{n-k+2}+2+2^{n-k+2}-2=2^n\]
    vertices, so this construction does in fact yield a caterpillar on \(2^n\) vertices. Suppose \(n,k>4\) and \(n\geq k\). Take some labeling \(w_1,f_1,w_2,f_2,\ldots,f_{2^{k-1}-1},w_{2^{k-1}}\) with \(k\)-dimensional vectors for the path \(P_{2^{k-1}}\). Append \(n-k+1\) zeros (denoted in general by \(0^{n-k+1}\), but here more succinctly as \(\vec{0}\)) to each vector to form vectors of dimension \((n+1)\), noting that this does not change the validity of the labeling. We will prove the theorem in two cases:
    \begin{itemize}
        \item[] Case 1: The additional \(2\cdot(2^{n-k+1}-1)\) pendent edges and vertices are added to one of the two end vertices of \(P_{2^{k-1}}\). Without loss of generality, suppose the pendent edges and vertices are added to the final vertex in the path. Let \(i=2^{n-k+1}-1, j = 2^{k-1}-2\), and label these pendent edges and vertices with \((n+1)\)-dimensional vectors \(v_1,e_1,\ldots,v_{j\cdot i+2\cdot(2^{n-k+1}-1)},e_{j\cdot i+2\cdot(2^{n-k+1}-1)}\) as in Figure \ref{fig:thm case 1}.

        \hspace{-1em}Since \(1+1=0\) under addition modulo \(2\), for any \(01\)-vector \(x\) of dimension \(n-k+1\), we have \(w_{\ell}x+f_{\ell}x=w_{\ell+1}\vec{0}\). There are \(2^{n-k+1}-1\) nonzero vectors we may construct using \(n-k+1\) symbols. Denote them by \(z_a\), for \(a=1,\ldots,2^{n-k+1}-1\). So we may express each vector \(w_{\ell}\vec{0}\), \(\ell\geq 2\) a total of \(2^{n-k+1}-1\) ways using the sums \(w_{\ell-1}z_a+f_{\ell-1}z_a=w_{\ell}\vec{0}\). Now for each of the interior vertices \(w_{\ell}\vec{0}\) of \(P_{2^{k-1}}\), let their pendent edges \(e_a\) and vertices \(v_a\) be given by \(e_a=f_{\ell-1}z_a\), \(v_a=w_{\ell-1}z_a\). We now have a set-sequential labeling using \((n+1)\)-dimensional vectors for \(w_1,f_1,\ldots,f_{2^{k-1}-1},w_{2^{k-1}}\) and for \(v_1,e_1,\ldots,v_{(j+1)\cdot i},e_{(j+1)\cdot i}\). What remains is to label \(v_{(j+1)\cdot i + 1},e_{(j+1)\cdot i + 1},\ldots v_{(j+2)\cdot i},e_{(j+2)\cdot i}\). 
        
        So far, we have not used the vectors \(w_{2^{k-1}}z_a\) or the \(0^kz_a\), for \(a=1,\ldots,2^{n-k+1}-1\). Let each of the remaining unlabeled pendent edges \(e_a\) and vertices \(v_a\) of \(w_{2^{k-1}}\vec{0}\) be given by \(e_a=0^kz_a\), \(v_a=w_{2^{k-1}}z_a\). Then for each \(v_b,e_b\) attached to \(w_{2^{k-1}}\vec{0}\), we have \(v_b+e_b=w_{2^{k-1}}\vec{0}\), finishing the required labeling for this caterpillar.
        
        \item[] Case 2: The additional \(2\cdot(2^{n-k+1}-1)\) pendent edges and vertices are added to one of the \(2^{k-1}-2\) interior vertices of \(P_{2^{k-1}}\), say, \(w_h\vec{0}\), for \(2\leq h\leq 2^{k-1}-1\). Let \(i=2^{n-k+1}-1, j = 2^{k-1}-2\), and label these pendent edges and vertices with \((n+1)\)-dimensional vectors \(v_1,e_1,\ldots,v_{j\cdot i+2\cdot(2^{n-k+1}-1)},e_{j\cdot i+2\cdot(2^{n-k+1}-1)}\) as in Figure \ref{fig:thm case 2}. Since \(1+1=0\) under addition modulo \(2\), for any \(01\)-vector \(x\) of dimension \(n-k+1\), we have \(w_{\ell}x+f_{\ell}x=w_{\ell+1}\vec{0}\). There are \(2^{n-k+1}-1\) nonzero vectors we may construct using \(n-k+1\) symbols. Denote them by \(z_a\), for \(a=1,\ldots,2^{n-k+1}-1\). So we may express each vector \(w_{\ell}\vec{0}\), \(2\leq\ell\leq2^{k-1}-1\), a total of \(2^{n-k+1}-1\) ways using either \(w_{\ell-1}z_a+f_{\ell-1}z_a=w_{\ell}\vec{0}\) or \(w_{\ell+1}z_a+f_{\ell}z_a=w_{\ell}\vec{0}\). Now for each of the interior vertices \(w_{\ell}\vec{0}\) of \(P_{2^{k-1}}\) with \(\ell\leq h-1\), let their pendent edges \(e_a\) and vertices \(v_a\) be given by \(e_a=f_{\ell-1}z_a\), \(v_a=w_{\ell-1}z_a\). For each of the interior vertices \(w_{\ell}\vec{0}\) with \(\ell\geq h+1\), let their pendent edges \(e_a\) and vertices \(v_a\) be given by \(e_a=f_{\ell}z_a\), \(v_a=w_{\ell+1}z_a\). We now have a good labeling using \((n+1)\)-dimensional vectors for \(w_1,f_1,\ldots,f_{2^{k-1}-1},w_{2^{k-1}}\) and for \(v_1,e_1,\ldots,v_{(h-2)\cdot i},e_{(h-2)\cdot i},v_{(h-1)\cdot i  +1},e_{(h-1)\cdot i +2},\ldots,v_{j\cdot i},e_{j\cdot i}\). What remains is to label \(v_{(h-2)\cdot i+1},e_{(h-2)\cdot i+1},\ldots,v_{(h-1)\cdot i},e_{(h-1)\cdot i}\) and \(v_{j\cdot i + 1},e_{j\cdot i + 1},\ldots v_{(j+2)\cdot i},e_{(j+2)\cdot i}\). 
        
        So far, the vectors \(w_{h-1}z_a\), \(f_{h-1}z_a\), \(w_{h}z_a\), \(f_hz_a\), \(w_{h+1}z_a\) or \(0^kz_a\), for \(a=1,\ldots,2^{n-k+1}-1\), have not been used. Since
        \(w_{h-1}z_a+f_{h-1}z_a=w_{h+1}z_a+f_hz_a=w_{h}z_a+0^kz_a=w_h\vec{0}\), let each pair of vectors \((v_{(h-2)\cdot i+1},e_{(h-2)\cdot i}+1),\ldots,(v_{(h-1)\cdot i},e_{(h-1)\cdot i})\) be equal to a pair \((w_{h-1}z_a,f_{h-1}z_a)\), let each pair of vectors \((v_{j\cdot i+1},e_{j\cdot i+1}),\ldots,(v_{(j+1)\cdot i},e_{(j+1)\cdot i})\) be equal to a pair \((w_{h+1}z_a,f_{h}z_a)\), and finally let each pair of vectors \((v_{(j+1)\cdot i+1},e_{(j+1)\cdot i+1}),\ldots,(v_{(j+2)\cdot i},e_{(j+2)\cdot i})\) be equal to a pair \((w_{h}z_a,0^kz_a)\). Then for each \(v_b,e_b\) attached to \(w_{h}\vec{0}\), we have \(v_b+e_b=w_{h}\vec{0}\), finishing the required labeling for this caterpillar.
    \end{itemize}
    \vspace{-1em}
\end{proof}

\begin{figure}[h]
        \begin{center}
            \begin{tikzpicture}
            \begin{scope}[every node/.style={circle,fill=blue,inner sep=0pt, minimum size = 1.5mm,draw}]
                \node (A) [label=below:{\textcolor{blue}{\(w_1\)}\(\vec{0}\)}] at (0,0) {};
                \node (B) [label=below:{\textcolor{blue}{\(w_2\)}\(\vec{0}\)}] at (1.5,0) {};
                \node (C) [label=above:{\textcolor{blue}{\(w_3\)}\(\vec{0}\)}] at (4.5,0) {};
                \node (D) at (7,0) {};
                \node (E) at (8,0) {};
                \node (G) [label={[below,xshift=-1em]:\textcolor{blue}{\(w_{2^{k-1}}\)}\(\vec{0}\)}] at (11.5,0) {};
            \end{scope}
            \begin{scope}[every node/.style={circle,fill=black,inner sep=0pt, minimum size = 1.5mm,draw}]
                \node (I) [label={\(v_1\)}] at (0,2) {};
                \node (J) [label={\(v_2\)}]  at (1,2) {};
                \node (K) [label={\(v_{i}\)}]  at (3,2) {};
                \node (L) [label={[below]:\(v_{i+1}\)}] at (3,-2) {};
                \node (M) [label={[below]:\(v_{i+2}\)}] at (4,-2) {};
                \node (N) [label={[below,yshift=-0.4em]:\(v_{2i}\)}] at (6,-2) {};
                \node (O) [label={[above,yshift=-0.5em,xshift=-1em]:\(v_{j\cdot i + 1}\)}] at (9.5,2) {};
                \node (P) [label={[above,xshift=1em,yshift=-0.5em]:\(v_{j\cdot i+2}\)}]  at (10.5,2) {};
                \node (Q) [label={[above,yshift=-0.8em]:\(v_{(j+1)\cdot i}\)}]  at (13.5,2) {};
                \node (R) [label={[right,yshift=0.5em,xshift=-2em]:\(v_{(j+1)\cdot i +1}\)}] at (14.5,0) {};
                \node (S) [label={[right,yshift=0.5em,xshift=-2em]:\(v_{(j+1)\cdot i +2}\)}] at (14.5,-1) {};
                \node (T) [label={[below,yshift=2em,xshift=-3em]:\(v_{(j+2)\cdot i}\)}] at (14.5,-3) {};
            \end{scope}
            \begin{scope}
                \node () at (2,2) {\(\cdots\)};
                \node () at (5,-2) {\(\cdots\)};
                \node () at (11.5,2) {\(\cdots\)};
                \node () at (14.5,-2) {\(\vdots\)};
            \end{scope}
            \begin{scope}[line width = 0.25mm]
                \path (A) edge node [label={[yshift=-0.5em]:\textcolor{blue}{\(f_1\)}\(\vec{0}\)}]{} (B);
                \path (B) edge node [label={[yshift=-0.5em]:\textcolor{blue}{\(f_2\)}\(\vec{0}\)}]{} (C);
                \path (C) edge node [label={[yshift=-0.5em]:\textcolor{blue}{\(f_3\)}\(\vec{0}\)}]{} (D);
                \path (D) [dashed] edge node {} (E);
                \path (E) edge node [label={[yshift=-0.5em]:\textcolor{blue}{\(f_{2^{k-1}-1}\)}\(\vec{0}\)}]{} (G);
                \path (G) edge node [label={[above,yshift=-0.5em]:\(e_{(j+1)\cdot i+1}\)}] {} (R);
                \path (G) edge node [fill=white, minimum size = 7mm,anchor=center, pos=0.5,label={[below,yshift=0.2em,xshift=0em,yshift=-0.5em]:\(e_{(j+1)\cdot i+2}\)}] {} (S);
                \path (G) edge node [label={[left,yshift=-1em]:\(e_{(j+2)\cdot i}\)}] {} (T);
                \path (G) edge node [label={[xshift=3.5em]:\(e_{(j+1)\cdot i}\)}] {} (Q);
                \path (G) edge node [label={[xshift=1em]:\(e_{j\cdot i+2}\)}] {} (P);
                \path (G) edge node [label={[xshift=-2.5em]:\(e_{j\cdot i+1}\)}] {} (O);
                \path (B) edge node [label={[left,xshift=-0.5em]:\(e_1\)}] {} (I);
                \path (B) edge node [label={[right]:\(e_2\)}] {} (J);
                \path (B) edge node [label={[right,xshift=0.5em]:\(e_i\)}] {} (K);
                \path (C) edge node [label={[left,yshift=-0.3em]:\(e_{i+1}\)}] {} (L);
                \path (C) edge node [fill=white,anchor=center, pos=0.5,label={[yshift=-1em]:\(e_{i+2}\)}] {} (M);
                \path (C) edge node [label={[right,yshift=-0.3em]:\(e_{2i}\)}] {} (N);
            \end{scope}
            \end{tikzpicture}
        \end{center}
        \caption{Labeling of the caterpillar in Case 1} \label{fig:thm case 1}
        \end{figure}
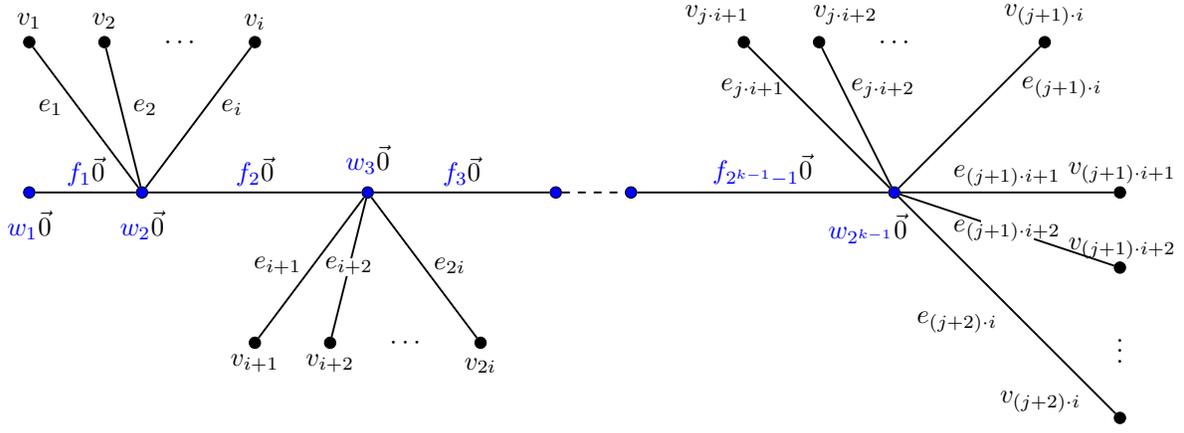
        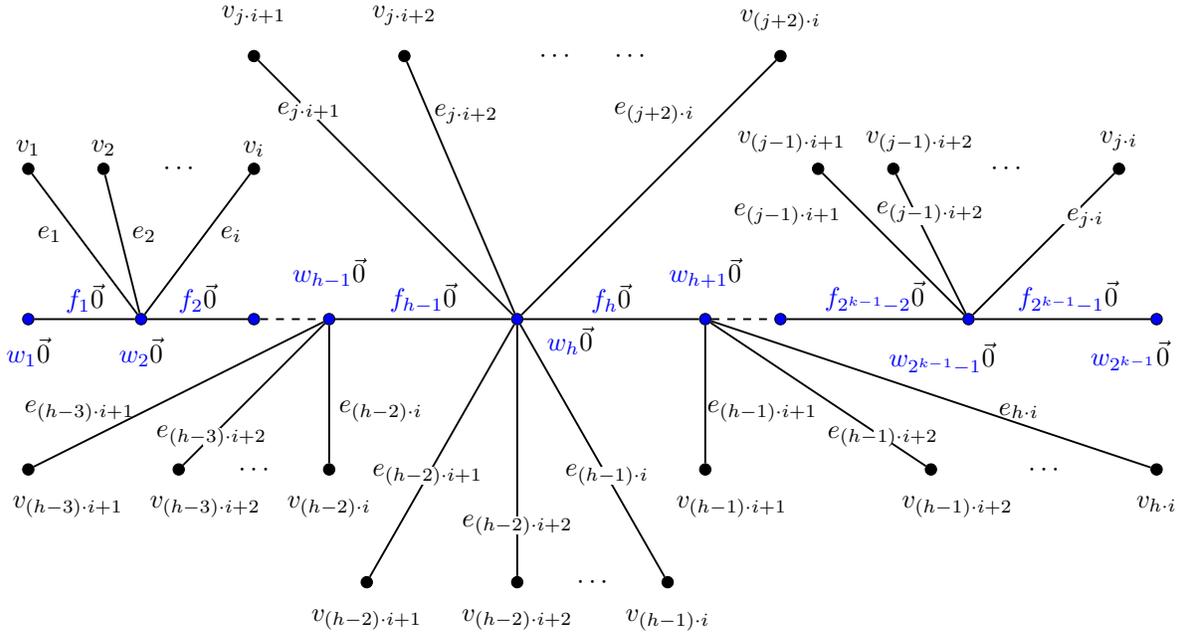
\begin{figure}[h!]
        \begin{center}
            \begin{tikzpicture}
            \begin{scope}[every node/.style={circle,fill=blue,inner sep=0pt, minimum size = 1.5mm,draw}]
                \node (A) [label=below:{\textcolor{blue}{\(w_1\)}\(\vec{0}\)}] at (0,0) {};
                \node (B) [label=below:{\textcolor{blue}{\(w_2\)}\(\vec{0}\)}] at (1.5,0) {};
                \node (C) at (3,0) {};
                \node (D) [label={[above]:\textcolor{blue}{\(w_{h-1} \)}\(\vec{0}\)}] at (4,0) {};
                \node (E) at (10,0) {};
                \node (F) [label={[above]:\textcolor{blue}{\(w_{h+1}\)}\(\vec{0}\)}] at (9,0) {};
                \node (G) [label={[below,xshift=-1em,yshift=0.4em]:\textcolor{blue}{\(w_{2^{k-1}-1}\)}\(\vec{0}\)}] at (12.5,0) {};
                \node (H) [label={[below,xshift=-1em]:\textcolor{blue}{\(w_{2^{k-1}}\)}\(\vec{0}\)}] at (15,0) {};
                \node (H1) [label={[below,yshift=0em,xshift=2em]:\textcolor{blue}{\(w_{h}\)}\(\vec{0}\)}] at (6.5,0) {};
            \end{scope}
            \begin{scope}[every node/.style={circle,fill=black,inner sep=0pt, minimum size = 1.5mm,draw}]
                \node (I) [label={\(v_1\)}] at (0,2) {};
                \node (J) [label={\(v_2\)}]  at (1,2) {};
                \node (K) [label={\(v_{i}\)}]  at (3,2) {};
                \node (O) [label={[above,yshift=-1.3em,xshift=-1em]:\(v_{(j-1)\cdot i + 1}\)}] at (10.5,2) {};
                \node (P) [label={[above,xshift=1em,yshift=-1.3em]:\(v_{(j-1)\cdot i+2}\)}]  at (11.5,2) {};
                \node (Q) [label={[above]:\(v_{j\cdot i}\)}]  at (14.5,2) {};
                \node (R) [label={[below,yshift=0em]:\(v_{(h-2)\cdot i}\)}] at (4,-2) {};
                \node (S) [label={[below,xshift=1em,yshift=0.5em]:\(v_{(h-3)\cdot i + 2}\)}] at (2,-2) {};
                \node (T) [label={[below,xshift=1.5em,yshift=0.5em]:\(v_{(h-3)\cdot i + 1}\)}] at (0,-2) {};
                \node (U) [label={[below,xshift=1em,yshift=0.5em]:\(v_{(h-1)\cdot i + 1}\)}] at (9,-2) {};
                \node (V) [label={[below,xshift=1em,yshift=0.5em]:\(v_{(h-1)\cdot i + 2}\)}] at (12,-2) {};
                \node (W) [label={[below,yshift=-0.7em]:\(v_{h\cdot i}\)}] at (15,-2) {};
                \node (X) [label={[below,yshift=0.5em]:\(v_{(h-2)\cdot i + 1}\)}] at (4.5,-3.5) {};
                \node (Y) [label={[below,yshift=0.5em]:\(v_{(h-2)\cdot i + 2}\)}] at (6.5,-3.5) {};
                \node (Z) [label={[below]:\(v_{(h-1)\cdot i}\)}] at (8.5,-3.5) {};
                \node (AA) [label={[above]:\(v_{j\cdot i + 1}\)}] at (3,3.5) {};  
                \node (AB) [label={[above]:\(v_{j\cdot i + 2}\)}] at (5,3.5) {};
                \node (AC) [label={[above,yshift=-0.5em]:\(v_{(j+2)\cdot i}\)}] at (10,3.5) {};
            \end{scope}
            \begin{scope}
                \node () at (2,2) {\(\cdots\)};
                \node () at (13,2) {\(\cdots\)};
                \node () at (3,-2) {\(\cdots\)};
                \node () at (13.5,-2) {\(\cdots\)};
                \node () at (7.5,-3.5) {\(\cdots\)};
                \node () at (8,3.5) {\(\cdots\)};
                \node () at (7,3.5) {\(\cdots\)};
            \end{scope}
            \begin{scope}[line width = 0.25mm]
                \path (A) edge node [label={[yshift=-0.5em]:\textcolor{blue}{\(f_1\)}\(\vec{0}\)}]{} (B);
                \path (B) edge node [label={[yshift=-0.5em]:\textcolor{blue}{\(f_2\)}\(\vec{0}\)}]{} (C);
                \path (C) [dashed] edge node {} (D);
                \path (F) [dashed] edge node {} (E);
                \path (E) edge node [label={[yshift=-0.5em]:\textcolor{blue}{\(f_{2^{k-1}-2}\)}\(\vec{0}\)}]{} (G);
                \path (G) edge node [label={[yshift=-0.5em,xshift=0.2em]:\textcolor{blue}{\(f_{2^{k-1}-1}\)}\(\vec{0}\)}]{} (H);
                \path (G) edge node [fill=white,anchor=center, pos=0.7,label={[yshift=-1.2em,xshift=0.5em]:\(e_{j\cdot i}\)}] {} (Q);
                \path (G) edge node [fill=white,anchor=center, pos=0.7,label={[yshift=-1em,xshift=0.5em]:\(e_{(j-1)\cdot i+2}\)}] {} (P);
                \path (G) edge node [label={[xshift=-4em]:\(e_{(j-1)\cdot i+1}\)}] {} (O);
                \path (B) edge node [label={[left,xshift=-0.5em]:\(e_1\)}] {} (I);
                \path (B) edge node [label={[right]:\(e_2\)}] {} (J);
                \path (B) edge node [label={[right,xshift=0.5em]:\(e_i\)}] {} (K);
                \path (D) edge node [label={[right,yshift=-1em]:\(e_{(h-2)\cdot i}\)}] {} (R);
                \path (D) edge node [fill=white, anchor=center, pos=0.8,label={[yshift=-1.1em]:\(e_{(h-3)\cdot i + 2}\)}] {} (S);
                \path (D) edge node [label={[left,yshift=-1em,xshift=-1.3em]:\(e_{(h-3)\cdot i + 1}\)}] {} (T);
                \path (F) edge node [label={[right,yshift=-1em,xshift=-0.3em]:\(e_{(h-1)\cdot i + 1}\)}] {} (U);
                \path (F) edge node [fill=white, anchor=center, pos=0.8,label={[yshift=-1.1em]:\(e_{(h-1)\cdot i + 2}\)}] {} (V);
                \path (F) edge node [label={[right,yshift=-1em,xshift=2.2em]:\(e_{h\cdot i }\)}] {} (W);
                \path (D) edge node [label={[yshift=-0.5em]:\textcolor{blue}{\(f_{h-1}\)}\(\vec{0}\)}] {} (H1);
                \path (H1) edge node [label={[yshift=-0.5em]:\textcolor{blue}{\(f_h\)}\(\vec{0}\)}] {} (F);
                \path (H1) edge node [fill=white,anchor=center, pos=0.6,label={[yshift=-1.1em]:\(e_{(h-2)\cdot i+1}\)}] {} (X);
                \path (H1) edge node [fill=white,anchor=center, pos=0.8,label={[yshift=-1.1em]:\(e_{(h-2)\cdot i+2}\)}] {} (Y);
                \path (H1) edge node [fill=white,anchor=center, pos=0.6,label={[yshift=-1.1em]:\(e_{(h-1)\cdot i}\)}] {} (Z);
                \path (H1) edge node [fill=white,anchor=center, pos=0.8,label={[yshift=-1.1em]:\(e_{j\cdot i+1}\)}] {} (AA);
                \path (H1) edge node [label={[yshift=1.7em,xshift=0.2em]:\(e_{j\cdot i+2}\)}] {} (AB);
                \path (H1) edge node [label={[yshift=1.7em,xshift=0.2em]:\(e_{(j+2)\cdot i}\)}] {} (AC);
            \end{scope}
            \end{tikzpicture}
        \end{center}
        \vspace{-1em}
        \caption{Labeling of the caterpillar in Case 2} \label{fig:thm case 2}
        \end{figure}
To allow for a more solid understanding of this proof, we present the following example.
\begin{example}
    Take \(n=6\), \(k=5\), and consider a caterpillar on \(64\) vertices, constructed in the manner of the above proof. We wish to show that it is set-sequential; that is, we wish to label it with the 127 nonzero \(01\)-vectors of dimension \(7\). The ``bone'' of the caterpillar is the path \(P_{16}\). We will add to each interior vertex of this path \(3\) pendent edges and vertices, and we will add to some other vertex in the path \(6\) additional pendent edges and vertices. This gives \(16\) vertices in the path, \(3\) extra vertices attached to \(14\) of the path vertices, and \(6\) extra vertices attached to some path vertex: in total, there are \(16+3\cdot 14+6=64\) vertices, with \(63\) edges. Take a labeling of \(P_{16}\) using the \(5\)-dimensional vectors \(w_1,f_2,\ldots,f_{15},w_{16}\). Figures \ref{fig:ex case 1} and \ref{fig:ex case 2} below show labelings for the caterpillar constructions in Cases 1 and 2 of the above proof, using 7-dimensional \(01\)-vectors.
    \begin{figure}[h!]
    \begin{center}
        \begin{tikzpicture}
        \begin{scope}[every node/.style={circle,fill=blue,inner sep=0pt, minimum size = 1.5mm,draw}]
            \node (A) [label=below:{\textcolor{blue}{\(w_1\)}00}] at (-1,0) {};
            \node (B) [label=below:{\textcolor{blue}{\(w_2\)}00}] at (1,0) {};
            \node (C) [label=above:{\textcolor{blue}{\(w_3\)}00}] at (4,0) {};
            \node (D) [label=below:{\textcolor{blue}{\(w_4\)}00}] at (7,0) {};
            \node (E) [label={[below,yshift=-0.3em,xshift=-0.2em]:\textcolor{blue}{\(w_{16}\)}00}] at (12,0) {};
            \node(F) [label=below:{\textcolor{blue}{\(w_5\)}00}] at (9,0) {};
            \node(G) [label={[below,yshift=-0.3em]:\textcolor{blue}{\(w_{15}\)}00}] at (10,0) {};
        \end{scope}
        \begin{scope}[every node/.style={circle,fill=black,inner sep=0pt, minimum size = 1.5mm,draw}]
            \node (J) [label=above:{\(w_1 01\)}]at (-1,2) {};
            \node (K) [label=above:{\(w_1 10\)}] at (1,2) {};
            \node (L) [label=above:{\(w_1 11\)}] at (3,2) {};
            \node (M) [label=below:{\(w_2 01\)}] at (2,-2) {};
            \node (N) [label=below:{\(w_2 10\)}] at (4,-2) {};
            \node (O) [label=below:{\(w_2 11\)}] at (6,-2) {}; 
            \node (P) [label=above:{\(w_3 01\)}] at (5,2) {};
            \node (Q) [label=above:{\(w_3 10\)}] at (7,2) {};
            \node (R) [label=above:{\(w_3 11\)}] at (9,2) {};
            \node (S) [label=above:{\(w_{15} 01\)}] at (10,2) {};
            \node (T) [label=above:{\(w_{15} 10\)}] at (12,2) {};
            \node (U) [label=above:{\(w_{15} 11\)}] at (14,2) {}; 
            \node (V) [label=right:{\(w_{16} 01\)}] at (14,0) {};
            \node (W) [label=right:{\(w_{16} 10\)}] at (14,-1) {};
            \node (X) [label=right:{\(w_{16} 11\)}] at (14,-2) {};
        \end{scope}
        \begin{scope}[line width = 0.25mm]
            \path (A) edge node [label={[yshift=-0.5em]:\textcolor{blue}{\(f_1\)}00}]{} (B);
            \path (B) edge node [label={[yshift=-0.5em]:\textcolor{blue}{\(f_2\)}00}]{} (C);
            \path (C) edge node [label={[yshift=-0.5em]:\textcolor{blue}{\(f_3\)}00}]{} (D);
            \path (D) edge node [label={[yshift=-0.5em]:\textcolor{blue}{\(f_4\)}00}]{} (F);
            \path (F) [dashed] edge node {} (G);
            \path (G) edge node [label={[yshift=-0.5em]:\textcolor{blue}{\(f_{15}\)}00}]{} (E);
            \path (B) edge node [label={\(f_1 01\)},yshift=0.5em] {} (J);
            \path (B) edge node [fill=white,minimum size=5mm,anchor=center, pos=0.6, label={[yshift=-1.5em]:\(f_1 10\)}] {} (K);
            \path (B) edge node [label={\(f_1 11\)},yshift=0.5em] {} (L);
            \path (C) edge node [label={\(f_2 01\)},yshift=-3em] {} (M);
            \path (C) edge node [fill=white,minimum size=5mm,anchor=center, pos=0.6, label={[yshift=-1.5em]:\(f_2 10\)}] {} (N);
            \path (C) edge node [label={\(f_2 11\)},yshift=-3em] {} (O);
            \path (D) edge node [label={\(f_3 01\)},yshift=0.5em] {} (P);
            \path (D) edge node [fill=white,minimum size=5mm,anchor=center, pos=0.6, label={[yshift=-1.5em]:\(f_3 10\)}] {} (Q);
            \path (D) edge node [label={\(f_3 11\)},yshift=0.5em] {} (R);
            \path (E) edge node [label={\(f_{15} 01\)},yshift=0.5em] {} (S);
            \path (E) edge node [fill=white,minimum size=5mm,anchor=center, pos=0.6, label={[yshift=-1.5em]:\(f_{15} 10\)}] {} (T);
            \path (E) edge node [label={\(f_{15} 11\)},yshift=0.5em] {} (U);
            \path (E) edge node [label={[yshift=-0.5em]:0000001}] {} (V);
            \path (E) edge node [fill=white,minimum size = 5mm,anchor=center, pos=0.6,label={[yshift=-1.4em,xshift=1em]:0000010}] {} (W);
            \path (E) edge node [fill=white,anchor=center, pos=0.6,label={[yshift=-1em]:0000011}] {} (X);
        \end{scope}
        \end{tikzpicture}
    \end{center}
    \caption{Labeling of the caterpillar in Case 1} \label{fig:ex case 1}
    \end{figure}
    \begin{figure}[h]
    \begin{center}
        \begin{tikzpicture}
        \begin{scope}[every node/.style={circle,fill=blue,inner sep=0pt, minimum size = 1.5mm,draw}]
            \node (A) [label={[below,yshift=-0.5em]:\(\textcolor{blue}{w_1} 00\)}] at (0,0) {};
            \node (B) [label={[below,yshift=-0.5em]:\(\textcolor{blue}{w_2} 00\)}] at (2,0) {};
            \node (C) at (4,0) {};
            \node (D) [label={[above,yshift=0.5em]:\(\textcolor{blue}{w_{10}}00\)}] at (5,0) {};
            \node (E) [label={[below,xshift=2.5em]:\(\textcolor{blue}{w_{11}}00\)}] at (7,0) {};
            \node (F) [label={[above,yshift=0.5em]:\(\textcolor{blue}{w_{12}}00\)}] at (9,0) {};
            \node (G) at (10,0) {};
            \node (H) [label={[below,yshift=-0.5em]:\(\textcolor{blue}{w_{15}}00\)}] at (12,0) {};
            \node (I) [label={[below,yshift=-0.5em]:\(\textcolor{blue}{w_{16}}00\)}] at (14,0) {};
        \end{scope}
        \begin{scope}[every node/.style={circle,fill=black,inner sep=0pt, minimum size = 1.5mm,draw}]
            \node (J) [label={[above]:\(w_1 01\)}] at (0,2) {};
            \node (K) [label={[above]:\(w_1 10\)}] at (2,2) {};
            \node (L) [label={[above]:\(w_1 11\)}] at (4,2) {};
            \node (M) [label={[below]:\(w_9 01\)}] at (1,-2) {};
            \node (N) [label={[below]:\(w_9 10\)}] at (3,-2) {};
            \node (O) [label={[below]:\(w_9 11\)}]at (5,-2) {};
            \node (P) [label={[below]:\(w_{10} 01\)}] at (5,-3.5) {};
            \node (Q) [label={[below]:\(w_{10} 10\)}] at (7,-3.5) {};
            \node (R) [label={[below]:\(w_{10} 11\)}] at (9,-3.5) {};
            \node (S) [label={[above]:\(w_{12} 01\)}] at (4,4) {};
            \node (T) [label={[above]:\(w_{12} 10\)}] at (5.2,4) {};
            \node (U) [label={[above]:\(w_{12} 11\)}] at (6.4,4) {};
            \node (V) [label={[above]:\(w_{11} 01\)}] at (7.6,4) {};
            \node (W) [label={[above]:\(w_{11} 10\)}] at (8.8,4) {};
            \node (X) [label={[above]:\(w_{11} 11\)}] at (10,4) {};
            \node (Y) [label={[below]:\(w_{13} 01\)}] at (9,-2) {};
            \node (Z) [label={[below]:\(w_{13} 10\)}] at (11,-2) {};
            \node (AA) [label={[below]:\(w_{13} 11\)}] at (13,-2) {};
            \node (AB) [label={[above]:\(w_{16} 01\)}] at (10,2) {};
            \node (AC) [label={[above]:\(w_{16} 10\)}] at (12,2) {};
            \node (AD) [label={[above]:\(w_{16} 11\)}] at (14,2) {};
        \end{scope}
        \begin{scope}[line width = 0.25mm]
            \path (A) edge node [label={[yshift=-0.5em]:\(\textcolor{blue}{f_1}00\)}] {} (B);
            \path (B) edge node [label={[yshift=-0.5em]:\(\textcolor{blue}{f_2}00\)}]{} (C);
            \path (C) [dashed] edge node {} (D);
            \path (D) edge node [label={[yshift=-0.5em]:\(\textcolor{blue}{f_{10}}00\)}] {} (E);
            \path (E) edge node [label={[yshift=-0.5em]:\(\textcolor{blue}{f_{11}}00\)}] {} (F);
            \path (F) [dashed] edge node {} (G);
            \path (G) edge node [label={[yshift=-0.5em]:\(\textcolor{blue}{f_{14}}00\)}] {} (H);
            \path (H) edge node [label={[yshift=-0.5em]:\(\textcolor{blue}{f_{15}}00\)}] {} (I);
            \path (B) edge node [label={[yshift=0.5em]:\(f_1 01\)}] {} (J);
            \path (B) edge node [fill=white, minimum size=5mm, anchor=center, pos=0.6,label={[yshift=-1.5em]:\(f_1 10\)}] {} (K);
            \path (B) edge node [label={[yshift=0.5em]:\(f_1 11\)}] {} (L);
            \path (D) edge node [label={[below,xshift=-3em]:\(f_9 01\)}] {} (M);
            \path (D) edge node [fill=white, minimum size=5mm, anchor=center, pos=0.6,label={[yshift=-1.5em]:\(f_9 10\)}] {} (N);
            \path (D) edge node [label={[right,yshift=-0.5em]:\(f_9 11\)}] {} (O);
            \path (E) edge node [fill=white, minimum size=5mm, anchor=center, pos=0.5,label={[yshift=-1.5em]:\(f_{10} 01\)}] {} (P);
            \path (E) edge node [fill=white, minimum size=5mm, anchor=center, pos=0.7,label={[yshift=-1.5em]:\(f_{10} 10\)}]{} (Q);
            \path (E) edge node [fill=white, minimum size=5mm, anchor=center, pos=0.5,label={[yshift=-1.5em]:\(f_{10} 11\)}] {} (R);
            \path (E) edge node [fill=white, minimum size=5mm, anchor=center, pos=0.8,label={[yshift=-1.5em]:\(f_{11} 01\)}] {} (S);
            \path (E) edge node [fill=white, minimum size=5mm, anchor=center, pos=0.7,label={[yshift=-1.5em]:\(f_{11} 10\)}] {} (T);
            \path (E) edge node [fill=white, minimum size=5mm, anchor=center, pos=0.6,label={[yshift=-1.5em]:\(f_{11} 11\)}] {} (U);
            \path (E) edge node [fill=white, minimum size=5mm, anchor=center, pos=0.7,label={[yshift=-1.5em]:\(00001\)}] {} (V);
            \path (E) edge node [fill=white, minimum size=5mm, anchor=center, pos=0.8,label={[yshift=-1.5em]:\(00010\)}] {} (W);
            \path (E) edge node [fill=white, minimum size=5mm, anchor=center, pos=0.9,label={[yshift=-1.5em]:\(00011\)}] {} (X);
            \path (F) edge node [label={[left,yshift=-0.5em]:\(f_{12} 01\)}] {} (Y);
            \path (F) edge node [fill=white, minimum size=5mm, anchor=center, pos=0.6,label={[yshift=-1.5em]:\(f_{12} 10\)}] {} (Z);
            \path (F) edge node [label={[below,xshift=3em]:\(f_{12} 11\)}] {} (AA);
            \path (H) edge node [label={[yshift=0.5em]:\(f_{15} 01\)}] {} (AB);
            \path (H) edge node [fill=white, minimum size=5mm, anchor=center, pos=0.6,label={[yshift=-1.5em]:\(f_{15} 10\)}] {} (AC);
            \path (H) edge node [label={[yshift=0.5em]:\(f_{15} 11\)}] {} (AD);
        \end{scope}
        \end{tikzpicture}
    \end{center}
    \caption{Labeling of the caterpillar in Case 2} \label{fig:ex case 2}
    \end{figure}
\end{example}

\subsection{Splicing}

Another technique that aids our goal of proving the Odd Tree Conjecture is one that constructs a large set-sequential tree from four smaller ones with an equal number of vertices. It is important that we start with \(4\) small trees rather than only two: Suppose we start with two copies \(T_1\) and \(T_2\) of a set-sequential tree on \(2^n\) vertices. From them we want to construct a tree on \(2^{n+1}\) vertices that is also set-sequential. In order to do this, we must take the vectors of dimension \(n+1\) and extend them by both a zero and a one in order to get all vectors of dimension \(n+2\) (except the vector with \(n+1\) zeros followed by a \(1\), denoted \(\vec{0}1\)). Express the trees as bipartite graphs with color classes \(X_1,Y_1\) and \(X_2,Y_2\) and edge sets \(E_1\) and \(E_2\), respectively. We have by the set-sequential nature of \(T_1\) and \(T_2\) that sums of the vectors labeling the vertices in \(X_i\) and the vectors labeling the vertices in \(Y_i\) are equal to the vectors labeling the edges in \(E_i\). Any extensions of these vector labels must preserve this, but if we only increase by one dimension, we may only extend with a zero or a one. It is not possible to do this: Extending the vectors labeling the vertices in \(X_1\) and \(Y_1\) and the edges in \(E_1\) by zero leaves that we must extend the vertices in \(X_2\) and \(Y_2\) and the edges in \(E_2\) by one, which produces not sums of zero but sums of \(\vec{0}1\). The only other option is, without loss of generality, to extend the vectors labeling the vertices in \(X_1\) and \(Y_1\) by one and the vectors labeling the edges in \(E_1\) by zero. This satisfies the sum condition for \(T_1\), but for \(T_2\) it gives again sums of \(\vec{0}1\). In order to avoid this, we consider constructing a tree on \(2^{n}\) vertices from four smaller trees on \(2^{n-2}\) vertices.

We therefore present an operation on 4 odd set-sequential trees on \(2^{k-2}\) vertices with equal bipartitions (that is, color classes must be of equal size, but edge sets may be different). We claim that this operation can be used to construct an odd set-sequential tree on \(2^k\) vertices. This operation utilizes the following definition, which we have named ``splicing.''

\begin{definition}
    Let \(G\) be a graph with some vertex \(v_1\) and some edge with endpoints \(u_1,u_2\). We define ``splicing \(v_1\) into \((u_1u_2)\)'' to be the operation that removes the edge \((u_1,u_2)\) and adds the edges \((v_1,u_1)\) and \((v_1,u_2)\).
\end{definition}

Consider three sets of splicing operations applied to four bipartite graphs with equal bipartitions (that is, the cardinalities of the color classes must be the same but the edge sets may be different):
\begin{enumerate}
    \item[] \textbf{Set 1}
    \begin{enumerate}[label=(\roman*)]
        \item \textit{Splice \(v_i\in X_i\) into \((u_{1,j}u_{2,j})\), where \(u_{1,j}\in X_j\), \(u_{2,j}\in Y_j\), and \(j\neq i\)}.
        
        \item \textit{Splice \(v_i\in X_i\) into \((u_{1,k}u_{2,k})\), where \(u_{1,k}\in X_k\), \(u_{2,k}\in Y_k\), and \(k\notin\{i,j\}\).}
        
        \item \textit{Splice \(v_i\in X_i\) into \((u_{1,\ell}u_{2,\ell})\), where \(u_{1,\ell}\in X_{\ell}\), \(u_{2,\ell}\in Y_{\ell}\), and \(\ell\notin\{i,j,k\}\).
        }
    \end{enumerate}
    \item[] \textbf{Set 2}
    \begin{enumerate}[label=(\roman*)]
        \item \textit{Splice \(v_i\in X_i\) into \((u_{1,j}u_{2,j})\), where \(u_{1,j}\in X_j\), \(u_{2,j}\in Y_j\), and \(j\neq i\).}
        \item \textit{Splice \(v_i\in X_i\) into \((u_{1,k}u_{2,k})\), where \(u_{1,k}\in X_k\), \(u_{2,k}\in Y_k\), and \(k\notin\{i,j\}\)}.
        \item \textit{Splice \(v_m\in Y_m\) into \((u_{1,\ell}u_{2,\ell})\), where \(u_{1,\ell}\in X_{\ell}\), \(u_{2,\ell}\in Y_{\ell}\), \(\ell\notin\{i,j,k\}\), and \(m\in\{j,k\}\).}
    \end{enumerate}
    
    \item[] \textbf{Set 3}
    \begin{enumerate}[label=(\roman*)]
        \item \textit{Splice \(x_i\in X_i\) into \((u_{1,j}u_{2,j})\), where \(u_{1,j}\in X_j\) and \(u_{2,j}\in Y_j\).}
        \item \textit{Splice \(v_j\in Y_j\) into \((u_{1,k}u_{2,k})\), where \(u_{1,k}\in X_k\), \(u_{2,k}\in Y_k\), and \(k\notin\{i,j\}\).}
        \item \textit{Splice \(v_k\in Y_k\) into \((u_{1,\ell}u_{2,\ell})\), where \(u_{1,\ell}\in X_{\ell}\), \(u_{2,\ell}\in Y_{\ell}\), and \(\ell\notin\{i,j,k\}\).}
    \end{enumerate}
\end{enumerate}

We claim that if the four bipartite graphs were odd set-sequential trees then the graph resulting from any of the sets of splicing operations is also an odd set-sequential tree. 

\begin{theorem}
\label{splicing}
    Take four odd set-sequential trees \(T_i\) on \(2^{k-2}\) vertices with color classes \(X_i\) and \(Y_i\), for \(i=1,2,3,4\). Suppose that \(|X_i|=|X_j|=\ell\) and \(|Y_i|=|Y_j|=m\) for \(i\neq j\). If there exists a labeling of each \(T_i\) such that \(\{v_1,v_2,\ldots,v_{\ell}\}\) labels each \(X_i\) and \(\{v_{\ell+1},v_{\ell+2},\ldots,v_{\ell+k}\}\) labels each \(Y_i\), then we may perform one of the sets of splicing operations given above to construct a set-sequential tree on \(2^k\) vertices.
\end{theorem}
\begin{proof} 
    The graphs resulting from the three splicing operations defined above are given in Figures \ref{fig:thm splice 1}-\ref{fig:thm splice 3} below, with the removed edge indicated in red. By examining the extensions of the color classes and the indicated edges connecting them, the graphs are seen to be set-sequential.
    \end{proof}

    \begin{figure}[h]
    \begin{center}
        \begin{tikzpicture}[scale=0.85]
            \draw[rounded corners=2.5ex] (0,0) rectangle ++(4,.9);
            \draw[rounded corners=2.5ex] (0,-1.5) rectangle ++(4,0.9);
            \draw[rounded corners=2.5ex] (7,0) rectangle ++(4,0.9);
            \draw[rounded corners=2.5ex] (7,-1.5) rectangle ++(4,0.9);
            \draw[rounded corners=2.5ex] (0,-4) rectangle ++(4,0.9);
            \draw[rounded corners=2.5ex] (0,-5.5) rectangle ++(4,0.9);
            \draw[rounded corners=2.5ex] (7,-4) rectangle ++(4,0.9);
            \draw[rounded corners=2.5ex] (7,-5.5) rectangle ++(4,0.9);
            \node at (-0.5,0.5) {10};
            \node at (-0.5,-0.3) {01};
            \node at (-0.5,-1.1) {11};
            \node at (-0.5,-3.5) {00};
            \node at (-0.5,-4.3) {00};
            \node at (-0.5,-5.1) {00};
            \node at (11.5,0.5) {01};
            \node at (11.5,-0.3) {11};
            \node at (11.5,-1.1) {10};
            \node at (11.5,-3.5) {11};
            \node at (11.5,-4.3) {10};
            \node at (11.5,-5.1) {01};
            \node at (3.5,-3.6) {\(X_1\)};
            \node at (3.5,-5.1) {\(Y_1\)};
            \node at (3.5,0.4) {\(X_2\)};
            \node at (3.5,-1.1) {\(Y_2\)};
            \node at (10.5,-3.6) {\(X_3\)};
            \node at (10.5,-5.1) {\(Y_3\)};
            \node at (10.5,0.4) {\(X_4\)};
            \node at (10.5,-1.1) {\(Y_4\)};
            \node at (2,-6) {\(T_1\)};
            \node at (2.25,1.25) {\(T_2\)};
            \node at (9,-6) {\(T_3\)};
            \node at (9,1.25) {\(T_4\)};
        \begin{scope}[every node/.style={circle,fill=black,inner sep=0pt, minimum size = 1.5mm,draw}]
            \node (A) [label={[right,xshift=0.5em]:\(v_1\)}] at (0.5,0.4) {};
            \node (B) [label={[right,xshift=0.5em]:\(v_2\)}] at (0.5,-1.1) {};
            \node (C) at (0.5,-3.6) {};
            \node (D) at (0.5,-5.1) {};
            \node (E) [label={[right,xshift=0.5em]:\(v_3\)}] at (7.5,0.4) {};
            \node (F) [label={[right,xshift=0.5em]:\(v_4\)}] at (7.5,-1.1) {};
            \node (G) [label={[right,xshift=0.5em]:\(v_5\)}] at (7.5,-3.6) {};
            \node (H) [label={[right,xshift=0.5em]:\(v_6\)}] at (7.5,-5.1) {};
            \node (I) [label={[below,yshift=-0.3em]:\(v_1\)}] at (1.25,-3.6) {};
            \node (J) [label={[below,yshift=-0.3em]:\(v_3\)}] at (2,-3.6) {};
            \node (K) [label={[below,yshift=-0.3em]:\(v_5\)}] at (2.75,-3.6) {};
        \end{scope}
        \begin{scope}[line width = 0.25mm,color=red,dashed]
            \path (A) edge node [label={[right]:\textcolor{red}{\(e\)}}] {} (B);
            \path (E) edge node [label={[left]:\textcolor{red}{\(e\)}}] {} (F);
            \path (G) edge node [label={[left]:\textcolor{red}{\(e\)}}] {} (H);
        \end{scope}
        \begin{scope}[line width = 0.25mm]
            \path (C) edge node [label={[right]\(e\)}] {} (D);
            \path (I) edge [bend left=30] node [label={[left,yshift=-2em]:\(\vec{0}10\)}] {} (A);
            \path (I) edge node [label={[right]:\(e11\)}] {} (B);
            \path (J) edge node [label={[above]:\(\vec{0}01\)}] {} (E);
            \path (J) edge node [label={[below,yshift=-0.5em]:\(e10\)}] {} (F);
            \path (K) edge [bend right=15] node [label={[above]:\(\vec{0}11\)}] {} (G);
            \path (K) edge [bend right=30]node [label={[below,yshift=-0.5em]:\(e01\)}] {} (H);
        \end{scope}
        \end{tikzpicture}
    \end{center}
    \vspace{-1em}
    \caption{Set 1 splicing operations} \label{fig:thm splice 1}
    \end{figure}
    
    \clearpage
    
    \begin{figure}[h]
    \begin{center}
        \begin{tikzpicture}[scale=0.85]
            \draw[rounded corners=2.5ex] (0,0) rectangle ++(4,0.9);
            \draw[rounded corners=2.5ex] (0,-1.5) rectangle ++(4,0.9);
            \draw[rounded corners=2.5ex] (7,0) rectangle ++(4,0.9);
            \draw[rounded corners=2.5ex] (7,-1.5) rectangle ++(4,0.9);
            \draw[rounded corners=2.5ex] (0,-4) rectangle ++(4,0.9);
            \draw[rounded corners=2.5ex] (0,-5.5) rectangle ++(4,0.9);
            \draw[rounded corners=2.5ex] (7,-4) rectangle ++(4,0.9);
            \draw[rounded corners=2.5ex] (7,-5.5) rectangle ++(4,0.9);
            \node at (-0.5,0.5) {10};
            \node at (-0.5,-0.3) {01};
            \node at (-0.5,-1.1) {11};
            \node at (-0.5,-3.5) {00};
            \node at (-0.5,-4.3) {00};
            \node at (-0.5,-5.1) {00};
            \node at (11.5,0.5) {01};
            \node at (11.5,-0.3) {11};
            \node at (11.5,-1.1) {10};
            \node at (11.5,-3.5) {11};
            \node at (11.5,-4.3) {10};
            \node at (11.5,-5.1) {01};
            \node at (3.5,-3.6) {\(X_1\)};
            \node at (3.5,-5.1) {\(Y_1\)};
            \node at (3.5,0.4) {\(X_2\)};
            \node at (3.5,-1.1) {\(Y_2\)};
            \node at (10.5,-3.6) {\(X_3\)};
            \node at (10.5,-5.1) {\(Y_3\)};
            \node at (10.5,0.4) {\(X_4\)};
            \node at (10.5,-1.1) {\(Y_4\)};
            \node at (2,-6) {\(T_1\)};
            \node at (2.25,1.25) {\(T_2\)};
            \node at (9,-6) {\(T_3\)};
            \node at (9,1.25) {\(T_4\)};
        \begin{scope}[every node/.style={circle,fill=black,inner sep=0pt, minimum size = 1.5mm,draw}]
            \node (A) [label={[right,xshift=0.5em]:\(v_1\)}] at (0.5,0.4) {};
            \node (B) [label={[right,xshift=0.5em]:\(v_2\)}] at (0.5,-1.1) {};
            \node (C) at (0.5,-3.6) {};
            \node (D) at (0.5,-5.1) {};
            \node (E) [label={[right,xshift=0.5em]:\(v_3\)}] at (7.5,0.4) {};
            \node (F) [label={[right,xshift=0.5em]:\(v_4\)}] at (7.5,-1.1) {};
            \node (G) [label={[right,xshift=0.5em]:\(v_5\)}] at (7.5,-3.6) {};
            \node (H) [label={[right,xshift=0.5em]:\(v_6\)}] at (7.5,-5.1) {};
            \node (I) [label={[below,yshift=-0.3em]:\(v_1\)}] at (1.25,-3.6) {};
            \node (J) [label={[below,yshift=-0.3em]:\(v_3\)}] at (2,-3.6) {};
            \node (K) [label={[right,xshift=0.5em]:\(v_6\)}] at (9,-1.1) {};
        \end{scope}
        \begin{scope}[line width = 0.25mm,color=red,dashed]
            \path (A) edge node [label={[right]:\textcolor{red}{\(e\)}}] {} (B);
            \path (E) edge node [label={[left]:\textcolor{red}{\(e\)}}] {} (F);
            \path (G) edge node [label={[left]:\textcolor{red}{\(e\)}}] {} (H);
        \end{scope}
        \begin{scope}[line width = 0.25mm]
            \path (C) edge node [label={[right]\(e\)}] {} (D);
            \path (I) edge [bend left=30] node [label={[left,yshift=-2em]:\(\vec{0}10\)}] {} (A);
            \path (I) edge node [label={[right]:\(e11\)}] {} (B);
            \path (J) edge node [label={[above]:\(\vec{0}01\)}] {} (E);
            \path (J) edge node [label={[below,yshift=-0.5em]:\(e10\)}] {} (F);
            \path (K) edge node [label={[left,xshift=-0.5em]:\(e01\)}] {} (G);
            \path (K) edge [bend left=15]node [label={[right,xshift=1em,yshift=2em]:\(\vec{0}11\)}] {} (H);
        \end{scope}
        \end{tikzpicture}
    \end{center}
    \vspace{-2em}
    \caption{Set 2 splicing operations} \label{fig:thm splice 2}
    \end{figure}
    \begin{figure}[h]
    \vspace{-0.5em}
    \begin{center}
        \begin{tikzpicture}[scale=0.85]
            \draw[rounded corners=2.5ex] (0,0) rectangle ++(4,0.9);
            \draw[rounded corners=2.5ex] (0,-1.5) rectangle ++(4,0.9);
            \draw[rounded corners=2.5ex] (7,0) rectangle ++(4,0.9);
            \draw[rounded corners=2.5ex] (7,-1.5) rectangle ++(4,0.9);
            \draw[rounded corners=2.5ex] (0,-4) rectangle ++(4,0.9);
            \draw[rounded corners=2.5ex] (0,-5.5) rectangle ++(4,0.9);
            \draw[rounded corners=2.5ex] (7,-4) rectangle ++(4,0.9);
            \draw[rounded corners=2.5ex] (7,-5.5) rectangle ++(4,0.9);
            \node at (-0.5,0.5) {10};
            \node at (-0.5,-0.3) {01};
            \node at (-0.5,-1.1) {11};
            \node at (-0.5,-3.5) {00};
            \node at (-0.5,-4.3) {00};
            \node at (-0.5,-5.1) {00};
            \node at (11.5,0.5) {01};
            \node at (11.5,-0.3) {11};
            \node at (11.5,-1.1) {10};
            \node at (11.5,-3.5) {11};
            \node at (11.5,-4.3) {10};
            \node at (11.5,-5.1) {01};
            \node at (3.5,-3.6) {\(X_1\)};
            \node at (3.5,-5.1) {\(Y_1\)};
            \node at (3.5,0.4) {\(X_2\)};
            \node at (3.5,-1.1) {\(Y_2\)};
            \node at (10.5,-3.6) {\(X_3\)};
            \node at (10.5,-5.1) {\(Y_3\)};
            \node at (10.5,0.4) {\(X_4\)};
            \node at (10.5,-1.1) {\(Y_4\)};
            \node at (2,-6) {\(T_1\)};
            \node at (2.25,1.25) {\(T_2\)};
            \node at (9,-6) {\(T_3\)};
            \node at (9,1.25) {\(T_4\)};
        \begin{scope}[every node/.style={circle,fill=black,inner sep=0pt, minimum size = 1.5mm,draw}]
            \node (A) [label={[right,xshift=0.5em]:\(v_1\)}] at (0.5,0.4) {};
            \node (B) [label={[right,xshift=0.5em]:\(v_2\)}] at (0.5,-1.1) {};
            \node (C) at (0.5,-3.6) {};
            \node (D) at (0.5,-5.1) {};
            \node (E) [label={[right,xshift=0.5em]:\(v_3\)}] at (7.5,0.4) {};
            \node (F) [label={[right,xshift=0.5em]:\(v_4\)}] at (7.5,-1.1) {};
            \node (G) [label={[right,xshift=0.5em]:\(v_5\)}] at (7.5,-3.6) {};
            \node (H) [label={[right,xshift=0.5em]:\(v_6\)}] at (7.5,-5.1) {};
            \node (I) [label={[below,yshift=-0.3em]:\(v_1\)}] at (1.25,-3.6) {};
            \node (J) [label={[below,yshift=-0.5em]:\(v_4\)}] at (2,-1.1) {};
            \node (K) [label={[right,xshift=0.5em]:\(v_6\)}] at (9,-1.1) {};
        \end{scope}
        \begin{scope}[line width = 0.25mm,color=red,dashed]
            \path (A) edge node [label={[right]:\textcolor{red}{\(e\)}}] {} (B);
            \path (E) edge node [label={[left]:\textcolor{red}{\(e\)}}] {} (F);
            \path (G) edge node [label={[left]:\textcolor{red}{\(e\)}}] {} (H);
        \end{scope}
        \begin{scope}[line width = 0.25mm]
            \path (C) edge node [label={[right]\(e\)}] {} (D);
            \path (I) edge [bend left=30] node [label={[left,yshift=-2em]:\(\vec{0}10\)}] {} (A);
            \path (I) edge node [label={[right]:\(e11\)}] {} (B);
            \path (J) edge node [label={[above]:\(e10\)}] {} (E);
            \path (J) edge [bend right=15] node [label={[below,yshift=-0.5em]:\(\vec{0}01\)}] {} (F);
            \path (K) edge node [label={[left,xshift=-0.5em]:\(e01\)}] {} (G);
            \path (K) edge [bend left=15]node [label={[right,xshift=1em,yshift=2em]:\(\vec{0}11\)}] {} (H);
        \end{scope}
        \end{tikzpicture}
    \end{center}
    \vspace{-2em}
    \caption{Set 3 splicing operations}\label{fig:thm splice 3}
    \end{figure}
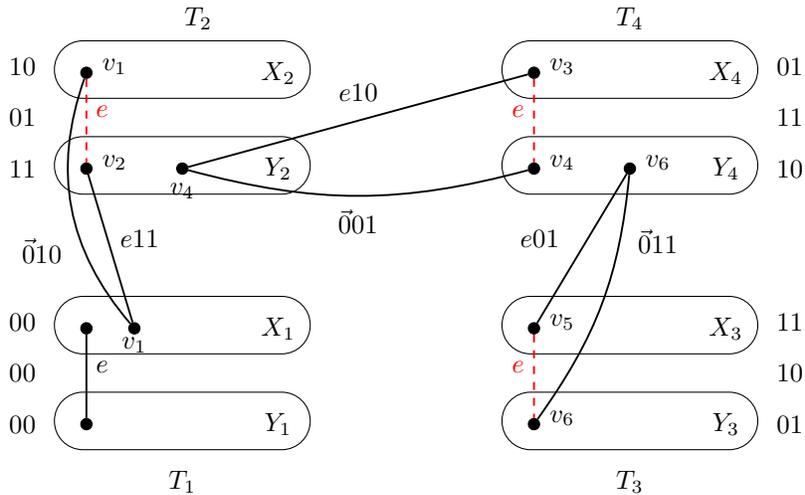

As an example of this splicing operation, we present in Figures \ref{fig:ex splice 1}-\ref{fig:ex splice 3}  some constructions of odd set-sequential trees on \(32\) vertices from four odd set-sequential trees on \(8\) vertices. Note here that these graphs are not caterpillars, nor can they be produced by the technique we will describe in the following section, so this splicing operation is indeed useful.
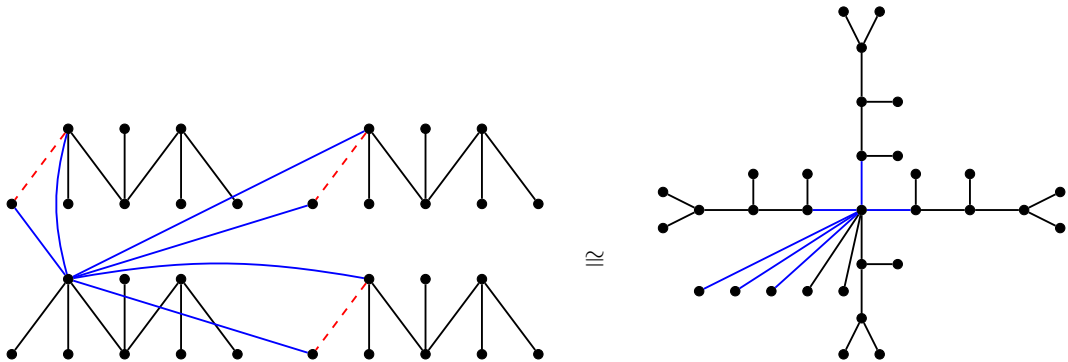
\begin{figure}[h!]
\vspace{-0.5em}
\begin{center}
    \begin{tikzpicture}[scale=0.5]
    \begin{scope}[every node/.style={circle,fill=black,inner sep=0pt, minimum size=1.25mm, draw}]
        \node (A1) at (0,0) {};
        \node (B1) at (1.5,0) {};
        \node (C1) at (3,0) {};
        \node (D1) at (-1.5,-2) {};
        \node (E1) at (0,-2) {};
        \node (F1) at (1.5,-2) {};
        \node (G1) at (3,-2) {};
        \node (H1) at (4.5,-2) {};
        \node (A2) at (8,0) {};
        \node (B2) at (9.5,0) {};
        \node (C2) at (11,0) {};
        \node (D2) at (6.5,-2) {};
        \node (E2) at (8,-2) {};
        \node (F2) at (9.5,-2) {};
        \node (G2) at (11,-2) {};
        \node (H2) at (12.5,-2) {};
        \node (A3) at (0,-4) {};
        \node (B3) at (1.5,-4) {};
        \node (C3) at (3,-4) {};
        \node (D3) at (-1.5,-6) {};
        \node (E3) at (0,-6) {};
        \node (F3) at (1.5,-6) {};
        \node (G3) at (3,-6) {};
        \node (H3) at (4.5,-6) {};
        \node (A4) at (8,-4) {};
        \node (B4) at (9.5,-4) {};
        \node (C4) at (11,-4) {};
        \node (D4) at (6.5,-6) {};
        \node (E4) at (8,-6) {};
        \node (F4) at (9.5,-6) {};
        \node (G4) at (11,-6) {};
        \node (H4) at (12.5,-6) {};
    \end{scope}
    \begin{scope}[line width = 0.25mm]
        \path (A1) edge node {} (E1);
        \path (A1) edge node {} (F1);
        \path (B1) edge node {} (F1);
        \path (C1) edge node {} (F1);
        \path (C1) edge node {} (G1);
        \path (C1) edge node {} (H1);
        \path (A2) edge node {} (E2);
        \path (A2) edge node {} (F2);
        \path (B2) edge node {} (F2);
        \path (C2) edge node {} (F2);
        \path (C2) edge node {} (G2);
        \path (C2) edge node {} (H2);
        \path (A3) edge node {} (D3);
        \path (A3) edge node {} (E3);
        \path (A3) edge node {} (F3);
        \path (B3) edge node {} (F3);
        \path (C3) edge node {} (F3);
        \path (C3) edge node {} (G3);
        \path (C3) edge node {} (H3);
        \path (A4) edge node {} (E4);
        \path (A4) edge node {} (F4);
        \path (B4) edge node {} (F4);
        \path (C4) edge node {} (F4);
        \path (C4) edge node {} (G4);
        \path (C4) edge node {} (H4);
    \end{scope}
    \begin{scope}[line width=0.25mm,color=red,dashed]
        \path (A1) edge node {} (D1);
        \path (A2) edge node {} (D2);
        \path (A4) edge node {} (D4);
    \end{scope}
    \begin{scope}[line width=0.25mm,color=blue]
        \path (A3) edge [bend left=15] node {} (A1);
        \path (A3) edge node {} (D1);
        \path (A3) edge node {} (A2);
        \path (A3) edge node {} (D2);
        \path (A3) edge [bend left=10] node {} (A4);
        \path (A3) edge node {} (D4);
    \end{scope}
    \node at (14,-3.5){\(\cong\)};
    \end{tikzpicture}
    \hspace{1em}
    \begin{tikzpicture}[scale=0.48]
    \begin{scope}[every node/.style={circle,fill=black,inner sep=0pt, minimum size=1.25mm, draw}]
        \node (A1) at (4,-14) {};
        \node (B1) at (2.5,-13) {};
        \node (C1) at (1,-14) {};
        \node (D1) at (3,-16.25) {};
        \node (E1) at (4,-13) {};
        \node (F1) at (2.5,-14) {};
        \node (G1) at (0,-13.5) {};
        \node (H1) at (0,-14.5) {};
        \node (A2) at (5.5,-12.5) {};
        \node (B2) at (6.5,-11) {};
        \node (C2) at (5.5,-9.5) {};
        \node (D2) at (2,-16.25) {};
        \node (E2) at (6.5,-12.5) {};
        \node (F2) at (5.5,-11) {};
        \node (G2) at (5,-8.5) {};
        \node (H2) at (6,-8.5) {};
        \node (A3) at (5.5,-14) {};
        \node (B3) at (6.5,-15.5) {};
        \node (C3) at (5.5,-17) {};
        \node (D3) at (4,-16.25) {};
        \node (E3) at (5,-16.25) {};
        \node (F3) at (5.5,-15.5) {};
        \node (G3) at (5,-18) {};
        \node (H3) at (6,-18) {};
        \node (A4) at (7,-14) {};
        \node (B4) at (8.5,-13) {};
        \node (C4) at (10,-14) {};
        \node (D4) at (1,-16.25) {};
        \node (E4) at (7,-13) {};
        \node (F4) at (8.5,-14) {};
        \node (G4) at (11,-13.5) {};
        \node (H4) at (11,-14.5) {};
    \end{scope}
    \begin{scope}[line width = 0.25mm]
        \path (A1) edge node {} (E1);
        \path (A1) edge node {} (F1);
        \path (B1) edge node {} (F1);
        \path (C1) edge node {} (F1);
        \path (C1) edge node {} (G1);
        \path (C1) edge node {} (H1);
        \path (A2) edge node {} (E2);
        \path (A2) edge node {} (F2);
        \path (B2) edge node {} (F2);
        \path (C2) edge node {} (F2);
        \path (C2) edge node {} (G2);
        \path (C2) edge node {} (H2);
        \path (A3) edge node {} (D3);
        \path (A3) edge node {} (E3);
        \path (A3) edge node {} (F3);
        \path (B3) edge node {} (F3);
        \path (C3) edge node {} (F3);
        \path (C3) edge node {} (G3);
        \path (C3) edge node {} (H3);
        \path (A4) edge node {} (E4);
        \path (A4) edge node {} (F4);
        \path (B4) edge node {} (F4);
        \path (C4) edge node {} (F4);
        \path (C4) edge node {} (G4);
        \path (C4) edge node {} (H4);
    \end{scope}
    \begin{scope}[line width=0.25mm,color=blue]
        \path (A3) edge node {} (A1);
        \path (A3) edge node {} (D1);
        \path (A3) edge node {} (A2);
        \path (A3) edge node {} (D2);
        \path (A3) edge node {} (A4);
        \path (A3) edge node {} (D4);
    \end{scope}
    \end{tikzpicture}
\end{center}
\vspace{-1em}
\caption{Set 1 splicing operations}\label{fig:ex splice 1}
\end{figure}
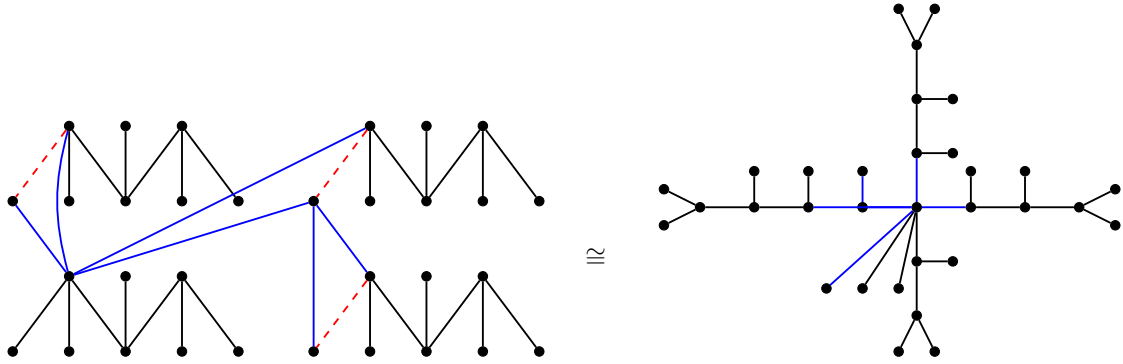
\begin{figure}[h!]
\begin{center}
\vspace{-1em}
    \begin{tikzpicture}[scale=0.5]
    \begin{scope}[every node/.style={circle,fill=black,inner sep=0pt, minimum size=1.25mm, draw}]
        \node (A1) at (0,0) {};
        \node (B1) at (1.5,0) {};
        \node (C1) at (3,0) {};
        \node (D1) at (-1.5,-2) {};
        \node (E1) at (0,-2) {};
        \node (F1) at (1.5,-2) {};
        \node (G1) at (3,-2) {};
        \node (H1) at (4.5,-2) {};
        \node (A2) at (8,0) {};
        \node (B2) at (9.5,0) {};
        \node (C2) at (11,0) {};
        \node (D2) at (6.5,-2) {};
        \node (E2) at (8,-2) {};
        \node (F2) at (9.5,-2) {};
        \node (G2) at (11,-2) {};
        \node (H2) at (12.5,-2) {};
        \node (A3) at (0,-4) {};
        \node (B3) at (1.5,-4) {};
        \node (C3) at (3,-4) {};
        \node (D3) at (-1.5,-6) {};
        \node (E3) at (0,-6) {};
        \node (F3) at (1.5,-6) {};
        \node (G3) at (3,-6) {};
        \node (H3) at (4.5,-6) {};
        \node (A4) at (8,-4) {};
        \node (B4) at (9.5,-4) {};
        \node (C4) at (11,-4) {};
        \node (D4) at (6.5,-6) {};
        \node (E4) at (8,-6) {};
        \node (F4) at (9.5,-6) {};
        \node (G4) at (11,-6) {};
        \node (H4) at (12.5,-6) {};
    \end{scope}
    \begin{scope}[line width = 0.25mm]
        \path (A1) edge node {} (E1);
        \path (A1) edge node {} (F1);
        \path (B1) edge node {} (F1);
        \path (C1) edge node {} (F1);
        \path (C1) edge node {} (G1);
        \path (C1) edge node {} (H1);
        \path (A2) edge node {} (E2);
        \path (A2) edge node {} (F2);
        \path (B2) edge node {} (F2);
        \path (C2) edge node {} (F2);
        \path (C2) edge node {} (G2);
        \path (C2) edge node {} (H2);
        \path (A3) edge node {} (D3);
        \path (A3) edge node {} (E3);
        \path (A3) edge node {} (F3);
        \path (B3) edge node {} (F3);
        \path (C3) edge node {} (F3);
        \path (C3) edge node {} (G3);
        \path (C3) edge node {} (H3);
        \path (A4) edge node {} (E4);
        \path (A4) edge node {} (F4);
        \path (B4) edge node {} (F4);
        \path (C4) edge node {} (F4);
        \path (C4) edge node {} (G4);
        \path (C4) edge node {} (H4);
    \end{scope}
    \begin{scope}[line width=0.25mm,color=red,dashed]
        \path (A1) edge node {} (D1);
        \path (A2) edge node {} (D2);
        \path (A4) edge node {} (D4);
    \end{scope}
    \begin{scope}[line width=0.25mm,color=blue]
        \path (A3) edge [bend left=15] node {} (A1);
        \path (A3) edge node {} (D1);
        \path (A3) edge node {} (A2);
        \path (A3) edge node {} (D2);
        \path (D2) edge node {} (A4);
        \path (D2) edge node {} (D4);
    \end{scope}
    \node at (14,-3.5){\(\cong\)};
    \end{tikzpicture}
    \hspace{1em}
    \begin{tikzpicture}[scale=0.48]
    \begin{scope}[every node/.style={circle,fill=black,inner sep=0pt, minimum size=1.25mm, draw}]
        \node (A1) at (2.5,-14) {};
        \node (B1) at (1,-13) {};
        \node (C1) at (-0.5,-14) {};
        \node (D1) at (3,-16.25) {};
        \node (E1) at (2.5,-13) {};
        \node (F1) at (1,-14) {};
        \node (G1) at (-1.5,-13.5) {};
        \node (H1) at (-1.5,-14.5) {};
        \node (A2) at (5.5,-12.5) {};
        \node (B2) at (6.5,-11) {};
        \node (C2) at (5.5,-9.5) {};
        \node (D2) at (4,-13) {};
        \node (E2) at (6.5,-12.5) {};
        \node (F2) at (5.5,-11) {};
        \node (G2) at (5,-8.5) {};
        \node (H2) at (6,-8.5) {};
        \node (A3) at (5.5,-14) {};
        \node (B3) at (6.5,-15.5) {};
        \node (C3) at (5.5,-17) {};
        \node (D3) at (4,-16.25) {};
        \node (E3) at (5,-16.25) {};
        \node (F3) at (5.5,-15.5) {};
        \node (G3) at (5,-18) {};
        \node (H3) at (6,-18) {};
        \node (A4) at (7,-14) {};
        \node (B4) at (8.5,-13) {};
        \node (C4) at (10,-14) {};
        \node (D4) at (4,-14) {};
        \node (E4) at (7,-13) {};
        \node (F4) at (8.5,-14) {};
        \node (G4) at (11,-13.5) {};
        \node (H4) at (11,-14.5) {};
    \end{scope}
    \begin{scope}[line width = 0.25mm]
        \path (A1) edge node {} (E1);
        \path (A1) edge node {} (F1);
        \path (B1) edge node {} (F1);
        \path (C1) edge node {} (F1);
        \path (C1) edge node {} (G1);
        \path (C1) edge node {} (H1);
        \path (A2) edge node {} (E2);
        \path (A2) edge node {} (F2);
        \path (B2) edge node {} (F2);
        \path (C2) edge node {} (F2);
        \path (C2) edge node {} (G2);
        \path (C2) edge node {} (H2);
        \path (A3) edge node {} (D3);
        \path (A3) edge node {} (E3);
        \path (A3) edge node {} (F3);
        \path (B3) edge node {} (F3);
        \path (C3) edge node {} (F3);
        \path (C3) edge node {} (G3);
        \path (C3) edge node {} (H3);
        \path (A4) edge node {} (E4);
        \path (A4) edge node {} (F4);
        \path (B4) edge node {} (F4);
        \path (C4) edge node {} (F4);
        \path (C4) edge node {} (G4);
        \path (C4) edge node {} (H4);
    \end{scope}
    \begin{scope}[line width=0.25mm,color=blue]
        \path (A3) edge node {} (A1);
        \path (A3) edge node {} (D1);
        \path (A3) edge node {} (A2);
        \path (D4) edge node {} (D2);
        \path (A3) edge node {} (A4);
        \path (A3) edge node {} (D4);
    \end{scope}
    \end{tikzpicture}
\end{center}
\vspace{-1em}
\caption{Set 2 splicing operations}\label{fig:ex splice 2}
\vspace{1em}
\end{figure}
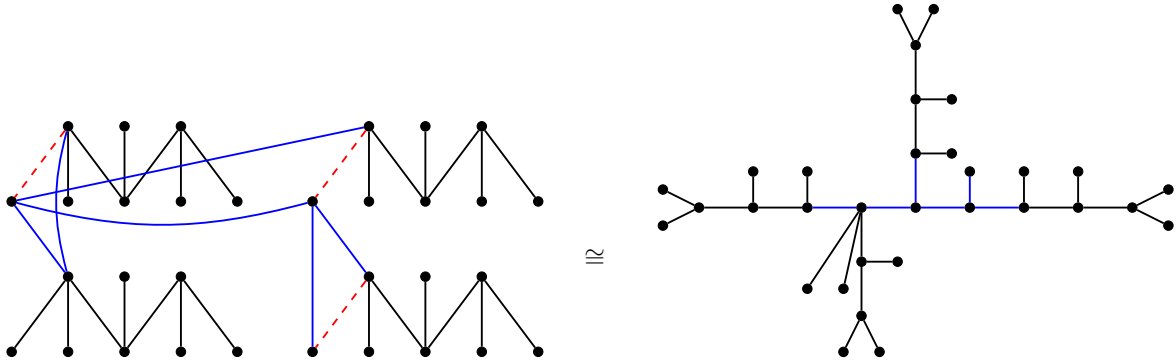
\begin{figure}[h!]
\begin{center}
    \begin{tikzpicture}[scale=0.5]
    \begin{scope}[every node/.style={circle,fill=black,inner sep=0pt, minimum size=1.25mm, draw}]
        \node (A1) at (0,0) {};
        \node (B1) at (1.5,0) {};
        \node (C1) at (3,0) {};
        \node (D1) at (-1.5,-2) {};
        \node (E1) at (0,-2) {};
        \node (F1) at (1.5,-2) {};
        \node (G1) at (3,-2) {};
        \node (H1) at (4.5,-2) {};
        \node (A2) at (8,0) {};
        \node (B2) at (9.5,0) {};
        \node (C2) at (11,0) {};
        \node (D2) at (6.5,-2) {};
        \node (E2) at (8,-2) {};
        \node (F2) at (9.5,-2) {};
        \node (G2) at (11,-2) {};
        \node (H2) at (12.5,-2) {};
        \node (A3) at (0,-4) {};
        \node (B3) at (1.5,-4) {};
        \node (C3) at (3,-4) {};
        \node (D3) at (-1.5,-6) {};
        \node (E3) at (0,-6) {};
        \node (F3) at (1.5,-6) {};
        \node (G3) at (3,-6) {};
        \node (H3) at (4.5,-6) {};
        \node (A4) at (8,-4) {};
        \node (B4) at (9.5,-4) {};
        \node (C4) at (11,-4) {};
        \node (D4) at (6.5,-6) {};
        \node (E4) at (8,-6) {};
        \node (F4) at (9.5,-6) {};
        \node (G4) at (11,-6) {};
        \node (H4) at (12.5,-6) {};
    \end{scope}
    \begin{scope}[line width = 0.25mm]
        \path (A1) edge node {} (E1);
        \path (A1) edge node {} (F1);
        \path (B1) edge node {} (F1);
        \path (C1) edge node {} (F1);
        \path (C1) edge node {} (G1);
        \path (C1) edge node {} (H1);
        \path (A2) edge node {} (E2);
        \path (A2) edge node {} (F2);
        \path (B2) edge node {} (F2);
        \path (C2) edge node {} (F2);
        \path (C2) edge node {} (G2);
        \path (C2) edge node {} (H2);
        \path (A3) edge node {} (D3);
        \path (A3) edge node {} (E3);
        \path (A3) edge node {} (F3);
        \path (B3) edge node {} (F3);
        \path (C3) edge node {} (F3);
        \path (C3) edge node {} (G3);
        \path (C3) edge node {} (H3);
        \path (A4) edge node {} (E4);
        \path (A4) edge node {} (F4);
        \path (B4) edge node {} (F4);
        \path (C4) edge node {} (F4);
        \path (C4) edge node {} (G4);
        \path (C4) edge node {} (H4);
    \end{scope}
    \begin{scope}[line width=0.25mm,color=red,dashed]
        \path (A1) edge node {} (D1);
        \path (A2) edge node {} (D2);
        \path (A4) edge node {} (D4);
    \end{scope}
    \begin{scope}[line width=0.25mm,color=blue]
        \path (A3) edge [bend left=15] node {} (A1);
        \path (A3) edge node {} (D1);
        \path (D1) edge node {} (A2);
        \path (D1) edge [bend right=15] node {} (D2);
        \path (D2) edge node {} (A4);
        \path (D2) edge node {} (D4);
    \end{scope}
    \node at (14,-3.5){\(\cong\)};
    \end{tikzpicture}
    \hspace{1em}
    \begin{tikzpicture}[scale=0.48]
    \begin{scope}[every node/.style={circle,fill=black,inner sep=0pt, minimum size=1.25mm, draw}]
        \node (A1) at (4,-14) {};
        \node (B1) at (2.5,-13) {};
        \node (C1) at (1,-14) {};
        \node (D1) at (8.5,-14) {};
        \node (E1) at (4,-13) {};
        \node (F1) at (2.5,-14) {};
        \node (G1) at (0,-13.5) {};
        \node (H1) at (0,-14.5) {};
        \node (A2) at (7,-12.5) {};
        \node (B2) at (8,-11) {};
        \node (C2) at (7,-9.5) {};
        \node (D2) at (8.5,-13) {};
        \node (E2) at (8,-12.5) {};
        \node (F2) at (7,-11) {};
        \node (G2) at (6.5,-8.5) {};
        \node (H2) at (7.5,-8.5) {};
        \node (A3) at (5.5,-14) {};
        \node (B3) at (6.5,-15.5) {};
        \node (C3) at (5.5,-17) {};
        \node (D3) at (4,-16.25) {};
        \node (E3) at (5,-16.25) {};
        \node (F3) at (5.5,-15.5) {};
        \node (G3) at (5,-18) {};
        \node (H3) at (6,-18) {};
        \node (A4) at (10,-14) {};
        \node (B4) at (11.5,-13) {};
        \node (C4) at (13,-14) {};
        \node (D4) at (7,-14) {};
        \node (E4) at (10,-13) {};
        \node (F4) at (11.5,-14) {};
        \node (G4) at (14,-13.5) {};
        \node (H4) at (14,-14.5) {};
    \end{scope}
    \begin{scope}[line width = 0.25mm]
        \path (A1) edge node {} (E1);
        \path (A1) edge node {} (F1);
        \path (B1) edge node {} (F1);
        \path (C1) edge node {} (F1);
        \path (C1) edge node {} (G1);
        \path (C1) edge node {} (H1);
        \path (A2) edge node {} (E2);
        \path (A2) edge node {} (F2);
        \path (B2) edge node {} (F2);
        \path (C2) edge node {} (F2);
        \path (C2) edge node {} (G2);
        \path (C2) edge node {} (H2);
        \path (A3) edge node {} (D3);
        \path (A3) edge node {} (E3);
        \path (A3) edge node {} (F3);
        \path (B3) edge node {} (F3);
        \path (C3) edge node {} (F3);
        \path (C3) edge node {} (G3);
        \path (C3) edge node {} (H3);
        \path (A4) edge node {} (E4);
        \path (A4) edge node {} (F4);
        \path (B4) edge node {} (F4);
        \path (C4) edge node {} (F4);
        \path (C4) edge node {} (G4);
        \path (C4) edge node {} (H4);
    \end{scope}
    \begin{scope}[line width=0.25mm,color=blue]
        \path (A3) edge node {} (A1);
        \path (A3) edge node {} (D4);
        \path (D4) edge node {} (A2);
        \path (D1) edge node {} (D2);
        \path (A4) edge node {} (D1);
        \path (D4) edge node {} (D1);
    \end{scope}
    \end{tikzpicture}
\end{center}
\vspace{-1em}
\caption{Set 3 splicing operations}\label{fig:ex splice 3}
\end{figure}

What follows now is an example showing that in Theorem \ref{splicing} the four smaller bipartite graphs must have the same labeling set for each upper and lower color class in order to extend a set-sequential labeling via the splicing operations defined.
\begin{example}
\label{color classes}
    Let \(R\) and \(S\) be the trees given in Figure \ref{fig:color classes}. Label \(R\) and \(S\) using the labelings given in Lemma \ref{odd tree}, noting that the vertices in the upper class of \(R\) are labeled with the vectors \(\{0001,1000,1001\}\) and the vertices in the upper class of \(S\) are labeled with the vectors \(\{0101,1011,1111\}\). Let \(T_1=T_2=T_3=R\) and \(T_4=S\). Before attempting any splicing, we must extend the vectors in each \(T_i\) by \(00\), \(01\), \(10\), or \(11\). Doing so, however, does not produce entirely distinct vectors. Extend the color classes \(X_1, Y_1\) of \(T_1\) by \(00\), \(X_2,Y_4\) by \(01\), \(X_3,Y_2\) by \(10\), and \(X_4,Y_3\) by \(11\) (or some similar arrangement). Then \(E_1\) is extended by \(00\), \(E_2\) by \(11\), \(E_3\) by \(01\), and \(E_4\) by \(10\). Extending in this way, we find \(0111\) extended by \(00\), \(10\), and \(11\), since it is in \(Y_1,Y_2,Y_3\). But in \(T_4\), \(0111\) is an edge, so it is extended by \(10\). This method of extension, then, gives that the vector \(011110\) is used twice in the larger tree on \(32\) vertices, which prevents that tree from being set-sequential. It is thus not enough that the bipartitons of each \(T_i\) are equal - we must assume that their upper and lower color classes have the same labeling set.
\end{example}
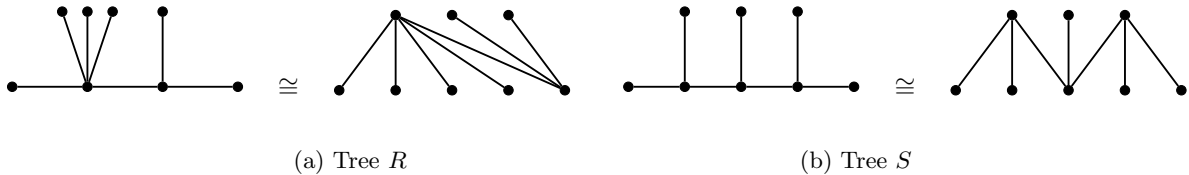
\begin{figure}[h!]
\begin{center}
    \begin{subfigure}[b]{0.4\textwidth}
    \begin{center}
    \begin{tabular}{c}
    \hspace{-5em}
        \begin{tikzpicture}[scale=0.5]
        \begin{scope}[every node/.style={circle,fill=black,inner sep=0pt, minimum size = 1.25mm,draw}]
            \node (A) at (-3,0) {};
            \node (B) at (-1,0) {};
            \node (C) at (1,0) {};
            \node (D) at (3,0) {};
            \node (E) at (1,2) {};
            \node (F) at (-1.67,2) {};
            \node (G) at (-1,2) {};
            \node (H) at (-0.33,2) {};
        \end{scope}
        \begin{scope}[line width = 0.25mm]
            \path (A) edge node {} (B);
            \path (B) edge node {} (C);
            \path (C) edge node {} (D);
            \path (B) edge node {} (F);
            \path (B) edge node {} (G);
            \path (B) edge node {} (H);
            \path (C) edge node {} (E);
        \end{scope}
        \end{tikzpicture}
        \(\quad\cong\quad\)
        \begin{tikzpicture}[scale=0.5]
        \begin{scope}[every node/.style={circle,fill=black,inner sep=0pt, minimum size = 1.25mm,draw}]
            \node (A) at (-3,0) {};
            \node (B) at (-1.5,0) {};
            \node (C) at (0,0) {};
            \node (D) at (1.5,0) {};
            \node (E) at (3,0) {};
            \node (F) at (-1.5,2) {};
            \node (G) at (0,2) {};
            \node (H) at (1.5,2) {};
        \end{scope}
        \begin{scope}[line width = 0.25mm]
            \path (A) edge node {} (F);
            \path (B) edge node {} (F);
            \path (C) edge node {} (F);
            \path (D) edge node {} (F);
            \path (E) edge node {} (F);
            \path (E) edge node {} (G);
            \path (E) edge node {} (H);
        \end{scope}
        \end{tikzpicture}
    \end{tabular}
    \end{center}
    \caption{Tree \(R\)} 
    \end{subfigure}
    \begin{subfigure}[b]{0.4\textwidth}
    \begin{center}
    \begin{tabular}{c}
        \begin{tikzpicture}[scale=0.5]
        \begin{scope}[every node/.style={circle,fill=black,inner sep=0pt, minimum size = 1.25mm,draw}]
            \node (A) at (-3,0) {};
            \node (B) at (-1.5,0) {};
            \node (C) at (0,0) {};
            \node (D) at (1.5,0) {};
            \node (E) at (3,0) {};
            \node (F) at (-1.5,2) {};
            \node (G) at (0,2) {};
            \node (H) at (1.5,2) {};
        \end{scope}
        \begin{scope}[line width = 0.25mm]
            \path (A) edge node {} (B);
            \path (B) edge node {} (C);
            \path (C) edge node {} (D);
            \path (D) edge node {} (E);
            \path (B) edge node {} (F);
            \path (C) edge node {} (G);
            \path (D) edge node {} (H);
        \end{scope}
        \end{tikzpicture}
        \(\quad\cong\quad\)
        \begin{tikzpicture}[scale=0.5]
        \begin{scope}[every node/.style={circle,fill=black,inner sep=0pt, minimum size = 1.25mm,draw}]
            \node (A) at (-3,0) {};
            \node (B) at (-1.5,0) {};
            \node (C) at (0,0) {};
            \node (D) at (1.5,0) {};
            \node (E) at (3,0) {};
            \node (F) at (-1.5,2) {};
            \node (G) at (0,2) {};
            \node (H) at (1.5,2) {};
        \end{scope}
        \begin{scope}[line width = 0.25mm]
            \path (A) edge node {} (F);
            \path (B) edge node {} (F);
            \path (C) edge node {} (F);
            \path (C) edge node {} (G);
            \path (C) edge node {} (H);
            \path (D) edge node {} (H);
            \path (E) edge node {} (H);
        \end{scope}
        \end{tikzpicture}
    \end{tabular}
    \end{center}
    \caption{Tree \(S\)}
    \end{subfigure}
\end{center}
\vspace{-1em}
\caption{Trees used in Example \ref{color classes}}
\label{fig:color classes}
\end{figure}
\clearpage
Though we can use this splicing technique to construct many odd trees, we cannot use it to produce all odd trees. 
\begin{counterexample}
    Consider an odd tree on \(n\) vertices with one vertex of degree of at least \(\frac{n}{2}\). If we take four small trees and perform one of the sets of splicing operations, the maximum degree in each small tree will be \(\frac{n}{4}\). We are adding only \(6\) edges to the graph, however. This means that we can not produce a vertex of degree \(\frac{n}{2}\) when \(n\) is large.
\end{counterexample}
How, then, can we produce labelings for graphs like the one in the above counterexample or for other graphs that cannot be obtained via splicing? We introduce the following section on partitions of \(\mathbb{F}_2^n\), motivated by work done in \cite{Balister}, to provide additional ways to construct odd set-sequential trees.

\subsection{Partitioning \(\mathbb{F}_2^n\)}

One conjecture that is of significant interest with respect to the Odd Tree Conjecture is Conjecture 1 in \cite{Balister}:

\begin{conjecture} \textnormal{(Balister et al.  \cite{Balister})} 
    \textit{Given \(2^{n-1}\) non-zero (not necessarily distinct) vectors \(v_1,\ldots,v_{2^{n-1}}\in\mathbb{F}_2^n\), \(n\geq 2\), with \(\sum_{i=1}^{2^{n-1}}v_i=0\), there exists a partition of \(\mathbb{F}_2^n\) into pairs of vectors \(\{p_i,q_i\}\), \(i=1,\ldots,2^{n-1}\) such that for all \(i\), \(v_i=p_i-q_i\)}.
\end{conjecture}

Some cases of this conjecture have been proven, in both \cite{Balister} and \cite{Golowich} (one result of particular note is that the conjecture is true for \(n\leq 5\)), though the general statement is yet unproven. This conjecture is of interest to our work because it gives us a way to begin with some odd set-sequential tree on \(2^n\) vertices and add \(2^n\) more vertices in such a way that the resulting tree on \(2^{n+1}\) vertices is still both odd and set-sequential: 
\begin{example}
    \label{base tree}
    Consider the tree \(T\) defined in Figure \ref{fig:ex tree}. This graph is not a caterpillar, nor can it be obtained via splicing, but we can use a special case of Conjecture 1 to prove that it is set-sequential. Note first that this \(T\) is one of the odd trees on \(8\) vertices with \(4\) vertices to which two additional vertices are joined. The odd tree on \(8\) vertices is given in blue in Figure \ref{fig:ex base tree}.
    
    In the larger \(16\) vertex tree, label \(v_1,\ldots,v_8\) using the labeling of that odd tree on \(8\) vertices (denoted \(k_1,\ldots,k_8\)) extended by \(0\). Even though Conjecture 1 in \cite{Balister} is yet unproven, it is known to be true for \(n\leq 5\). So we may partition \(\mathbb{F}_2^4\) into quadruples of vectors \((p_i,q_i,\ell_i,m_i\)) such that \(k_i=p_i+q_i={\ell}_i+m_i\), for \(i=1,2,3,4\). (This is accomplished by letting \(k_1=k_2\), \(k_3=k_4\), \(k_5=k_6\), and \(k_7=k_8\) in Conjecture 1 of \cite{Balister}.) Without loss of generality, label the \(8\) appended vertices with \(p_i1\) and \(\ell_i1\) and the edges with \(q_i1\) and \(m_i1\). Then since \(k_i=p_i+q_i=\ell_i+m_i\), we have \(k_i0=p_i1+q_i1=\ell_i1+m_i1\). We know that the sum condition for the 8-vertex subgraph is satisfied as well since we have merely extended it by \(0\). Therefore \(T\) is set-sequential.
    \begin{figure}[h]
    \begin{center}
        \begin{subfigure}[b]{0.3\textwidth}
            \begin{tikzpicture}[scale=0.5]
            \begin{scope}[every node/.style={circle,fill=black,inner sep=0pt, minimum size = 1.25mm,draw}]
                \node (A) at (0,0) {};
                \node (B) at (2,0) {};
                \node (C) at (4,0) {};
                \node (D) at (6,0) {};
                \node (E) at (8,0) {};
                \node (F) at (0,2) {};
                \node (G) at (2,2) {};
                \node (H) at (4,2) {};
                \node (I) at (6,2) {};
                \node (J) at (8,2) {};
                \node (K) at (0,4) {};
                \node (L) at (2,4) {};
                \node (M) at (3,4) {};
                \node (N) at (5,4) {};
                \node (O) at (6,4) {};
                \node (P) at (8,4) {};
            \end{scope}
            \begin{scope}[line width=0.25mm]
                \path (A) edge node {} (B);
                \path (B) edge node {} (C);
                \path (C) edge node {} (D);
                \path (D) edge node {} (E);
                \path (B) edge node {} (F);
                \path (C) edge node {} (G);
                \path (C) edge node {} (H);
                \path (C) edge node {} (I);
                \path (D) edge node {} (J);
                \path (G) edge node {} (K);
                \path (G) edge node {} (L);
                \path (H) edge node {} (M);
                \path (H) edge node {} (N);
                \path (I) edge node {} (O);
                \path (I) edge node {} (P);
            \end{scope}
            \end{tikzpicture}
        \caption{Tree \(T\)}\label{fig:ex tree}
        \end{subfigure}
        \begin{subfigure}[b]{0.3\textwidth}
            \begin{tikzpicture}[scale=0.5]
            \begin{scope}[every node/.style={circle,fill=black,inner sep=0pt, minimum size = 1.25mm,draw}]
                \node (A) at (0,0) {};
                \node (F) at (0,2) {};
                \node (K) at (0,4) {};
                \node (L) at (2,4) {};
                \node (M) at (3,4) {};
                \node (N) at (5,4) {};
                \node (O) at (6,4) {};
                \node (P) at (8,4) {};
            \end{scope}
            \begin{scope}[every node/.style={circle,fill=blue,inner sep=0pt, minimum size = 1.25mm,draw}]
                \node (B) [label={[right,xshift=0.5em,yshift=0.5em]:\(v_1\)}] at (2,0) {};
                \node (C) at (4,0) {};
                \node (D) at (6,0) {};
                \node (E) at (8,0) {};
                \node (G) [label={[right,xshift=0.5em,yshift=-0.5em]:\(v_2\)}] at (2,2) {};
                \node (H) [label={[right,xshift=0.5em,yshift=-0.5em]:\(v_3\)}] at (4,2) {};
                \node (I) [label={[right,xshift=0.5em,yshift=-0.5em]:\(v_4\)}] at (6,2) {};
                \node (J) at (8,2) {};
            \end{scope}
            \begin{scope}[line width=0.25mm]
                \path (A) edge node {} (B);
                \path (B) edge node {} (C);
                \path (C) edge node {} (D);
                \path (D) edge node {} (E);
                \path (B) edge node {} (F);
                \path (C) edge node {} (G);
                \path (C) edge node {} (H);
                \path (C) edge node {} (I);
                \path (D) edge node {} (J);
                \path (G) edge node {} (K);
                \path (G) edge node {} (L);
                \path (H) edge node {} (M);
                \path (H) edge node {} (N);
                \path (I) edge node {} (O);
                \path (I) edge node {} (P);
            \end{scope}
            \end{tikzpicture}
        \caption{The 8-vertex base tree of \(T\)}\label{fig:ex base tree}
        \end{subfigure}
    \end{center}
    \caption{Trees used in Example \ref{base tree}}
    \end{figure}
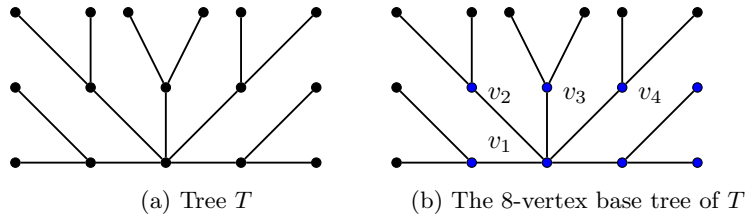
    
    To construct an explicit labeling of the \(T\) in the proof above, we start from the labeling of the \(8\)-vertex subgraph as given in the proof of Lemma \ref{odd tree}. Then, using the vertex labels given in the above proof, we have \(k_1=0011\), \(k_2=1011\), \(k_3=0101\), and \(k_4=1101\). Even though the construction leading to Conjecture 1 in \cite{Balister} was not given, through some trial and error, and following in some ways the style of the proof of Conjecture 2 in the same work, we can obtain that \(0011=0100+0111=0101+0110\), \(1011=0000+1011=0001+1010\), \(0101=1000+1101=1001+1100\), and \(1101=0010+1111=0011+1110\). It is an easy exercise to verify that the vectors in the sums form a partition of \(\mathbb{F}_2^4\). So label the appended vertices with one of the summands in each of the \(8\) sums and the edge with the other, both extended by \(1\).
\end{example}

This special case of Conjecture 2.9 given above seems to be of significant use, so we endeavor to prove it for all \(n\) (we know now only that it is true for \(n\leq 5\)). Note here that we chose this particular case in order to ensure oddness is preserved. Since applying the method shown above adds in all cases an even number of pendent edges to a vertex of presently odd degree, the resulting larger tree still contains vertices of only odd degree.
\begin{conjecture}[Pairing Conjecture]
    Take nonzero vectors \(v_1,v_2,\ldots,v_{2^{n-1}}\in\mathbb{F}_2^n\), where \(n\geq 2\), and \(v_{2i+1}=v_{2i+2}\), for \(i=0,\ldots,2^{n-2}-1\). We may partition \(\mathbb{F}_2^n\) into pairs \((p_i,q_i)\) such that \(v_i=p_i+q_i\) for all \(i\).
\end{conjecture}

A useful case of the Pairing Conjecture (which is also a case of Conjecture 1 in \cite{Balister}) is the following, which appeared as Theorem 4 in \cite{Balister}: 

\begin{theorem} \textnormal{(Balister, Gy\H{o}ri, Schelp \cite{Balister})}
    \textit{Given \(2^{n-1}\) non-zero vectors \(v_1,\ldots,v_{2^{n-1}}\in\mathbb{F}_2^n\), \(n\geq 2\), with \(v_1=v_2=\cdots=v_{2^{n-2}}\) and \(v_{2i+1}=v_{2i+2}\) for all \(i=0,\ldots,2^{n-2}-1\), there exists a partition of \(\mathbb{F}_2^n\) into pairs of vectors \(\{p_i,q_i\}\), \(i=1,\ldots,2^{n-1}\) such that for all \(i\), \(v_i=p_i-q_i\).}
\end{theorem}

The proof of this theorem that appeared in \cite{Balister} only dealt with a single (rather neat) case and did not provide details as to how one might alter the proof to account for the other cases. We present here a proof of Theorem 4 that uses the same basic idea as the proof in \cite{Balister}, but we do so more rigorously in an effort to eliminate any confusion regarding its validity.

\begin{proof}
    This proof contains similar ideas to those found in the proof in \cite{Balister}, though it allows for \(v_1\) to be chosen to be some vector other than \(\vec{0}1\). First partition \(\mathbb{F}_2^n\) into \(2^{n-1}\) pairs \(\{r_m,s_m\}\) so that \(v_1=r_m+s_m\) for \(m=1,\ldots,2^{n-1}\). (If \(n=2\), set \(p_1=r_1,q_1=s_1,p_2=r_2,q_2=s_2\), and we are done.) Observe that for \(j=2^{n-2}\), we may express \(v_{j+1}\) as \(2^{n-1}\) distinct sums, so set \(p_{j+1},q_{j+1},p_{j+2},q_{j+2}\) according to the following cases:
    \begin{enumerate}
        \item \(v_{j+1}=r_a+r_b,a\neq b\)\\
        Then we have 
        \[v_{j+2}=v_{j+1}=r_a+r_b=(v_1+s_a)+(v_1+s_b)=s_a+s_b,\]
        so let \(p_{j+1}=r_a\), \(q_{j+1}=r_b\), \(p_{j+2}=s_a\), \(q_{j+2}=s_b\).
        \item \(v_{j+1}=s_a+s_b,a\neq b\)\\
        Then, similarly to the previous case, we have
        \[v_{j+2}=v_{j+1}=s_a+s_b=(v_1+r_a)+(v_1+r_b)=r_a+r_b,\]
        so let \(p_{j+1}=s_a\), \(q_{j+1}=s_b\), \(p_{j+2}=r_a\), \(q_{j+2}=r_b\).
        \item \(v_{j+1}=r_a+s_b,a\neq b\)\\
        Then we have
        \[v_{j+2}=v_{j+1}=r_a+s_b=(v_1+s_a)+(v_1+r_b)=s_a+r_b,\]
        so let \(p_{j+1}=r_a\), \(q_{j+1}=s_b\), \(p_{j+2}=s_a\), \(q_{j+2}=r_b\).
        \item \(v_{j+1}=r_a+s_a\)\\
        Then, for some \(b\neq a\), we have
        \[v_{j+2}=v_{j+1}=r_a+s_a=r_b+s_b,\]
        so let \(p_{j+1}=r_a\), \(q_{j+1}=s_a\), \(p_{j+2}=r_b\), \(q_{j+2}=s_b\).
    \end{enumerate}
    In all cases, then, we have \(\{p_{j+1},q_{j+1},p_{j+2},q_{j+2}\}=\{r_a,s_a,r_b,s_b\}\), \(a\neq b\). (If \(n=3\), we are done.)
    
    Consider now the next pair of equal vectors, \(v_{j+3},v_{j+4}\). As in the case of \(v_{j+1}=v_{j+2}\), we may express \(v_{j+3}\) as \(2^{n-1}\) distinct sums. Here, as many as four sums may contain one of the vectors \(r_a,s_a,r_b,s_b\), so since \(2^{n-1}-4>0\), for \(n\geq4\), by the same process as above, we may assign \(p_{j+3},q_{j+3},p_{j+4},q_{j+4}\) so that \(\{p_{j+3},q_{j+3},p_{j+4},q_{j+4}\}=\{r_c,s_c,r_d,s_d\}\), where \(c\) and \(d\) are distinct and are not equal to \(a\) or \(b\). (If \(n=4\), we are done.)
    
    The next pair of equal vectors, \(v_{j+5},v_{j+6}\) can be again expressed as \(2^{n-1}\) distinct sums, but this time as many as \(8\) sums may contain one of \(r_a,s_a,r_b,s_b,r_c,s_c,r_d,s_d\). In a similar manner to the above, since \(2^{n-1}-8>0\) for \(n\geq 5\), we may assign \(p_{j+5},q_{j+5},p_{j+6},q_{j+6}\) so that \(\{p_{j+5},q_{j+5},p_{j+6},q_{j+6}\}=\{r_e,s_e,r_f,s_f\}\) for \(e\) and \(f\) distinct and not equal to \(a,b,c\) or \(d\).
    
    After each assignment of the pairs \(\{p_k,q_k\}\) and \(\{p_{k+1},q_{k+1}\}\), we have \(4\) fewer ``available'' sums for the remaining pairs of vectors. The last pair of vectors for which we need to assign \(p_k,q_k\) is \(v_{2^{n-1}-1}=v_{2^{n-1}}\). By this time, we have \(2^{n-1}-4\cdot(2^{n-3}-1)\) possible sums not containing vectors used in a previous partitioning step, since each of the \(2^{n-3}\) pairs after the first reduces the number of ``available'' sums by \(4\). Now \(2^{n-1}-4\cdot(2^{n-3}-1)=2^{n-1}-2^{n-1}+4>0\), so it is still possible to set \(p_{2^{n-1}-1},q_{2^{n-1}-1},p_{2^{n-1}},q_{2^{n-1}}\) so that \(\{p_{2^{n-1}-1},q_{2^{n-1}-1},p_{2^{n-1}},q_{2^{n-1}}\}=\{r_y,s_y,r_z,s_z\}\) for \(y\) and \(z\) distinct and not equal to any of the \(a,b,c,d,e,f,\ldots\) used previously. 
    
    There remain \(2^{n-2}\) pairs \(\{r_k,s_k\}\) such that \(v_1=r_k+s_k\) that have not yet been assigned as some pair \(\{p_{\ell},q_{\ell}\}\). So let each pair \(\{p_{m},q_{m}\}\) for \(m=1,\ldots,2^{n-2}\) be given by one of the pairs \(\{r_k,s_k\}\). We have now partitioned \(\mathbb{F}_2^n\) as desired.
\end{proof}

The collection of techniques and constructions given here allows us to make substantial progress on the Odd Tree Conjecture, but most graphs cannot be obtained using one of the operations or constructions we have given. For example, the graph in Figure \ref{fig:unobtainable} is not a caterpillar, nor can it be obtained via either the splicing method we proposed or the method that arises from partitioning \(\mathbb{F}^n_2\) in any way.
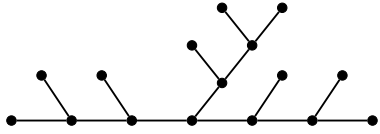
\begin{figure}[h]
\begin{center}
    \begin{tikzpicture}[scale=0.4]
    \begin{scope}[every node/.style={circle,fill=black,inner sep=0pt, minimum size = 1.25mm,draw}]
        \node (A) at (0,0) {};
        \node (B) at (2,0) {};
        \node (C) at (4,0) {};
        \node (D) at (6,0) {};
        \node (E) at (8,0) {};
        \node (F) at (10,0) {};
        \node (G) at (12,0) {};
        \node (H) at (1,1.5) {};
        \node (I) at (3,1.5) {};
        \node (J) at (7,1.25) {};
        \node (K) at (9,1.5) {};
        \node (L) at (11,1.5) {};
        \node (M) at (6,2.5) {};
        \node (N) at (8,2.5) {};
        \node (O) at (7,3.75) {};
        \node (P) at (9,3.75) {};
    \end{scope}
    \begin{scope}[line width=0.25mm]
        \path (A) edge node {} (B);
        \path (B) edge node {} (C);
        \path (C) edge node {} (D);
        \path (D) edge node {} (E);
        \path (E) edge node {} (F);
        \path (F) edge node {} (G);
        \path (B) edge node {} (H);
        \path (C) edge node {} (I);
        \path (D) edge node {} (J);
        \path (E) edge node {} (K);
        \path (F) edge node {} (L);
        \path (J) edge node {} (M);
        \path (J) edge node {} (N);
        \path (N) edge node {} (O);
        \path (N) edge node {} (P);
    \end{scope}
    \end{tikzpicture}
\end{center}
\caption{A tree that cannot be obtained via the preceding methods}\label{fig:unobtainable}
\end{figure}
\section{Additional Results Involving Bipartite Graphs}

In addition to pursuing the proof of the Odd Tree Conjecture, we also studied a technique presented in \cite{Balister} that involves joining four copies of a set-sequential bipartite graph (not necessarily an odd tree) in such a way as to produce a larger graph that is set-sequential. This raises the following question: Given four copies of a set-sequential bipartite graph, how can we join them by three single edges so that the resulting graph is set-sequential?
\begin{lemma}
    Let \(G\) be a set-sequential graph on \(n\) vertices that is either a caterpillar or a bipartite graph with at least two vertices of degree \(1\) in each color class. Let \(G_1,G_2,G_3,G_4\) be copies of \(G\). In each \(G_i\), choose two pendent vertices \(v_i\) and \(u_i\), \(i=1,2,3,4\), that are in the same color class. The graph constructed by adding the edges \((v_1,v_2)\), \((u_2,u_3)\), and \((v_3,v_4)\) is set-sequential.
\end{lemma}
\begin{proof}
    Note that caterpillars are bipartite graphs, so we may express \(G\) as a bipartite graph with color classes \(X\) and \(Y\). Take four copies \(G_1,G_2,G_3,G_4\) of \(G\) with color classes \(X_1,Y_1,X_2,Y_2,X_3,Y_3,X_4,Y_4\), respectively. We must extend the \((n+1)\)-dimensional vectors labeling \(G\) by two digits so they are of the proper dimension to form a set-sequential labeling for the larger graph consisting of \(G_1,G_2,G_3,G_4\) with three edges added. Extend the vectors in \(X_1\) and \(Y_1\) by ``\(00\)'', the vectors in \(X_2\) and \(Y_4\) by ``\(11\)'', the vectors in \(X_3\) and \(Y_2\) by ``\(10\)'', and the vectors in \(X_4\) and \(Y_3\) by ``\(01\)''. Note that these extensions produce all nonzero vectors of dimension \((n+3)\) with the exceptions of \(\vec{0}01\), \(\vec{0}10\), and \(\vec{0}11\). The edge \((v_1,v_2)\) will be labeled with \(\vec{0}11\), the edge \((u_2,u_3)\) will be labeled with \(\vec{0}01\), but the edge \((v_3,v_4)\) will also be labeled with \(\vec{0}11\), which does not fulfill the requirements for a set-sequential labeling. We may alter this labeling to be set-sequential, however: The vector labeling \(v_3\) ends in ``\(10\)'', and the vector labeling its pendent edge ends in ``\(11\)''. Since \(v_3\) is a leaf, we may ``swap'' the last two digits of these vectors with each other without violating the sum condition that is necessary for a set-sequential labeling. Now then the edge \((v_3,v_4)\) is labeled with \(\vec{0}10\), so the graph consisting of \(G_1,G_2,G_3,G_4\) together with the three indicated edges is has a labeling and so is set-sequential.
\end{proof}

We give now a modification of Theorem 3 in \cite{Balister}, imposing two extra conditions that allow us to present a more complete proof of the theorem than was originally given in \cite{Balister}.

\begin{theorem}
    \textit{Let \(G\) be a strongly set colorable bipartite graph with color classes \(X\), \(Y\) and edge set \(E\). Let \(G_1,G_2,G_3,G_4\) be four disjoint copies of \(G\) with color classes \(X_1,Y_1,X_2,Y_2,X_3,Y_3,X_4,Y_4\) and edge sets \(E_1,E_2,E_3,E_4\), respectively. Let \(G_0\) denote the graph obtained from the disjoint union of the graphs \(G_1,G_2,G_3,G_4\) by adding edges \(e_1,e_2,e_3\) with the following four properties: 
    \begin{enumerate}
        \item each \(e_i\) joins two copies of the same vertex;
        \item one of the following three possibilities occurs:
        \begin{enumerate}
        \item the edges join \(X_1\) and \(X_2\), \(X_2\) and \(X_3\), \(X_3\) and \(X_4\), respectively; or
        \item the edges join \(X_1\) and \(Y_2\), \(X_2\) and \(Y_3\), \(X_3\) and \(Y_4\), respectively; or
        \item the edges join \(X_1\) and \(X_2\), \(Y_2\) and \(Y_4\), \(Y_1\) and \(Y_3\), respectively.
        \end{enumerate}
        \item there are two pendent vertices in \(G\), namely \(u_1,u_2,u_3,u_4,v_1,v_2,v_3,v_4\) in corresponding \(G_1,G_2,G_3,G_4\)
        \item all the edges are joining leaves in partite sets.
    \end{enumerate}
    Then \(G_0\) is strongly set colorable.}
\end{theorem}
\begin{proof}
For Case (a), consider extensions \(``00"\) to \(X_1,Y_1\), \(``11"\) on \(X_2,Y_4\), \(``10"\) on \(Y_2,Y_3\) and \(``01"\) on \(X_3,X_4\). This is a set-sequential labeling for all vertices and edges except for the edge from \(X_3\) to \(X_4\). As in the proof of Lemma 3.1, switch the extension of the vector labeling the vertex with connecting edges to \(X_3\) in \(X_4\) with the extension of the vector labeling its pendent edge in \(X_4\). 

For Case (b), consider extensions \(``00"\) to \(X_1,Y_1\), \(``11"\) on \(X_2,Y_4\), \(``10"\) on \(Y_2,Y_3\) and \(``01"\) on \(X_3,X_4\). This is a set-sequential labeling for all vertices and edges except for the edge from \(X_3\) to \(Y_4\). As before, switch the extension of the vector labeling the vertex with connecting edges to \(X_3\) in \(Y_4\) with the extension of the vector labeling its pendent edge in \(Y_4\). 

For Case (c), consider extensions \(``00"\) to \(X_1,Y_1\), \(``11"\) on \(X_2,Y_4\), \(``10"\) on \(Y_2,Y_3\) and \(``01"\) on \(X_3,X_4\). This is a set-sequential labeling for all vertices and edges except the edge from \(Y_2\) to \(Y_4\). Again, switch the extension of the vector labeling the vertex with connecting edges to \(Y_2\) in \(Y_4\) with the extension of the vector labeling its edge pendent in \(Y_2\). 
\end{proof}

\section{Conclusion}

One object of immediate interest is resolving the Pairing Conjecture. With that proof in hand, we would then have found a large class of odd trees that is set-sequential.  This combined with our results on caterpillars and splicing small trees together represents substantial progress toward the Odd Tree Conjecture. As explained at the end of Section 2, though, this set of constructions does not prove the Odd Tree Conjecture in its entirety. Further study may produce additional constructions and techniques that would allow us to join smaller trees together in ways that preserve both oddness and set-sequentialness. This may still fall short of proving the Odd Tree Conjecture as presently stated, but it is thought that larger classes of set-sequential trees may be obtained in this manner.

\section{Acknowledgements}
The work presented here was done as part of the Budapest Semesters in Mathematics Summer Undergraduate Research Program under the supervision of the second author. The work of the second author was supported by the National Research, Development and Innovation Office under Grant K132696. The first, third, and fourth authors are grateful to the second for presenting and allowing joint work on such an interesting problem.

\end{document}